\documentclass[12pt]{article}

\usepackage{epsf,epsfig,amsmath,amsfonts,amsthm}
\usepackage{graphicx}
\usepackage{color}
\usepackage{accents}
\usepackage{mathrsfs}
\usepackage{tikz}
\usepackage{hyperref}

\newtheorem{theorem}{Theorem}

\newtheorem{lemma}[theorem]{Lemma}

\newtheorem{remark}{Remark}
\newtheorem{assumption}{Assumption}
\newtheorem{example}{Example}

\usepackage{graphics}
\usepackage{epsfig}
\usepackage{lipsum}
\usepackage{amsfonts}
\usepackage{amssymb}
\usepackage{graphicx}
\usepackage{epstopdf}
\usepackage{algorithm}
\usepackage{algorithmic}
\usepackage{graphicx}
\usepackage{enumitem}
\usepackage{tikz}
\usepackage{booktabs}
\usepackage{float}
\usetikzlibrary{positioning,arrows.meta,shadows,fit}

\usepackage{subfigure}
\ifpdf
  \DeclareGraphicsExtensions{.eps,.pdf,.png,.jpg}
\else
  \DeclareGraphicsExtensions{.eps}
\fi

\usepackage{enumitem}

\def\x{{\bf x}}
\def\y{{\bf y}}
\def\z{{\bf z}}

\usepackage{bm}
\usepackage{mathrsfs}

\newcommand{\argmin}{\operatorname*{arg\,min}}

\usepackage{tikz}
\usetikzlibrary{arrows.meta,positioning,calc}

\usepackage{xcolor}

\usepackage{cases}
\usepackage{lipsum}
\usepackage{amsfonts}
\usepackage{graphicx}
\usepackage{epstopdf}
\usepackage{algorithmic}

\usepackage{amsopn}

\usepackage{amsmath}
\usepackage{booktabs}
\usepackage{graphicx}
\usepackage{hyperref}
\usepackage{siunitx}
\usepackage{subcaption}
\usepackage{xcolor}

\begin{document}

\title{Stability of Differential Stochastic Variational Inequalities with History-Dependent Responses and Transfer Learning}

\author{
{Xiaojun Chen\footnote{Department of Applied Mathematics, The Hong Kong Polytechnic University, Kowloon, Hong Kong
(\texttt{maxjchen@polyu.edu.hk}, \texttt{jiguo@polyu.edu.hk}, \texttt{guan1.wang@polyu.edu.hk}).},
\ Jian Guo\footnotemark[1],
\ Xin Guo\footnote{Department of Industrial Engineering and Operations Research, University of California, Berkeley, CA, USA
(\texttt{xinguo@berkeley.edu}).}
\ and Guan Wang\footnotemark[1]
}}

\maketitle
\begin{abstract} \noindent
In this paper, we propose and study a class of differential stochastic variational inequalities (DSVIs), in which an ordinary differential equation (ODE) is coupled with history-dependent stochastic variational
inequalities (SVI).
This framework models closed-loop stochastic systems with time-varying random
equilibria and includes optimization-constrained ODEs as special cases.
Under appropriate technical conditions, we establish uniqueness, measurability, and Lipschitz continuity with respect to the state of the second-stage response, and consequently the existence and uniqueness of the induced state trajectory. 
 Moreover, we construct a sample average approximation (SAA)  based on independent sample
paths and prove uniform  convergence of the approximate trajectories. 
For transfer between related stochastic environments, we derive a local $1/2$-Hölder  estimate for parametric variational inequalities with moving feasible sets and a quantitative trajectory-stability bound in terms of the initial-state difference and the Wasserstein distance between exogenous path laws. 
Numerical experiments illustrate the SAA convergence and transfer-stability results. We further apply the framework to an elderly-health monitoring system. Similarity-weighted reuse of precomputed responses achieves an accuracy close to the full-recomputation benchmark of 0.97, while reducing the online batch runtime from 86 seconds to less than one second. Perturbation and delayed-update experiments additionally characterize robustness to sensor noise and the trade-off between response freshness, predictive accuracy, and computational cost. These results provide theoretical and computational support for efficient transfer learning in history-dependent DSVI systems.
\end{abstract}


\medskip


\section{Introduction}

\paragraph{\bf Motivating example in health monitoring.} Consider a health-monitoring system for the elderly that continuously combines smartwatch, intelligent-insole, and electronic medical record data to assess a patient’s evolving health condition. The current health state cannot generally be inferred from the latest measurement alone: cumulative activity, gait patterns, clinical history, delayed physiological effects, and adaptive feedback make the assessment intrinsically history dependent. At the same time,  a data-driven response is repeatedly estimated from the observed history and fed back into the state dynamics. For a newly enrolled patient with limited reference data, it is also desirable to reuse response trajectories learned from similar patients; however, differences in initial health conditions, sensor streams, and underlying data distributions may propagate through the resulting closed-loop system. These features naturally motivate a history-dependent differential stochastic variational inequality framework that couples projected state dynamics with stochastic variational responses. With this system come several central theoretical questions: well-posedness, sample-based approximation, and stability under source-to-target perturbations. 

\paragraph{\bf DSVI.} 
 To start, let $T>0$, $\mathcal T=[0,T]$, and $\mathcal X\subseteq\mathbb{R}^n$ be a nonempty,
closed, and convex set, and
$(\Omega,\mathcal{F},(\mathcal{F}_t)_{t\ge 0},\mathbb{P})$
be a filtered probability space satisfying the usual conditions, that is,
(i) $\mathcal{F}_0$ contains all subsets of $\mathbb{P}$-null sets in $\mathcal{F}$, and
(ii) $(\mathcal{F}_t)_{t\ge 0}$ is right-continuous:
$\mathcal{F}_t=\bigcap_{s>t}\mathcal{F}_s$ for all $t\ge 0$.
Let $
\bm{\xi}=(\bm{\xi}_t)_{t\in\mathcal T}
$ be an \((\mathcal{F}_t)\)-progressively measurable process taking values in a Borel set \(\Xi\subseteq\mathbb{R}^{\ell}\).
Moreover,  for each
$t\in\mathcal T$ and $\omega\in\Omega$, define its history up to time $t$ by
$\bm{\xi}_{\cdot\wedge t}(\omega):[0,t]\to\Xi$ with
$\bigl(\bm{\xi}_{\cdot\wedge t}(\omega)\bigr)(s)
:=\bm{\xi}_s(\omega)$ for $s\in[0,t]$.
For fixed $\omega\in\Omega$, we write
$\xi_s:=\bm{\xi}_s(\omega)$ and
$\xi_{\cdot\wedge t}:=\bm{\xi}_{\cdot\wedge t}(\omega)$.

With these notations, we propose and  study the following history-dependent differential
stochastic variational inequality system (DSVI):
\begin{subequations}\label{eq:closed_loop_model_projected_random}
\begin{numcases}{}
\begin{aligned}
\dot x(t)
&=
\Pi_{\mathcal X}\Bigl(
x(t)
-
\mathbb E\bigl[
\Phi\bigl(
t,\bm{\xi}_t,x(t),y(t,\bm{\xi}_{\cdot\wedge t})
\bigr)
\bigr]
\Bigr)
-
x(t),
\end{aligned}
&\hspace{-15pt} $t\in\mathcal T$,
\label{eq:closed_loop_model_projected_random_first}
\\
\begin{aligned}
y(t,\bm{\xi}_{\cdot\wedge t})
&\in
\operatorname{SOL}
\bigl(
R,F,\mathcal Y(t,\bm{\xi}_t);
t,\bm{\xi}_{\cdot\wedge t},x(t)
\bigr),
\end{aligned}
&\hspace{-15pt}  $t\in\mathcal T,\ \mathrm{a.s.}$,
\label{eq:closed_loop_model_projected_random_second}
\\
x(0)=\x_0.
& \nonumber
\end{numcases}
\end{subequations}
Here  $\Pi_{\mathcal X}(\cdot)$ denotes the Euclidean projection onto  ${\mathcal X}$, {$
\Phi:\mathbb R_+\times\Xi\times\mathbb R^n\times\mathbb R^m
\to\mathbb R^n
$ is a continuous mapping,} and $\x_0\in {\mathcal X}$ is the
initial state. $
\mathcal Y:\mathbb R_+\times\Xi\rightrightarrows\mathbb R^m
$
is a set-valued mapping with nonempty closed convex values such that
$
\mathcal Y(t,\xi)\subseteq\overline{\mathcal Y}
$
for all $(t,\xi)\in\mathbb R_+\times\Xi$, where
$\overline{\mathcal Y}\subset\mathbb R^m$ is compact.
And 
$
\operatorname{SOL}
\bigl(
R,F,\mathcal Y(t,\xi_t);
t,\xi_{\cdot\wedge t},\x
\bigr)
$
is the solution set of the following variational inequality (VI):

Given a fixed $t\in\mathcal T$, a history path
$
\xi_{\cdot\wedge t}:s\in[0,t]\mapsto\xi_s\in\Xi,
$
a state $\x\in\mathbb R^n$, and  two continuous mappings
$
R:\mathbb R_+\times\Xi\times\mathbb R^n\times\mathbb R^m
\to\mathbb R^m
$ and 
$
F:\mathbb R_+\times\Xi\times\mathbb R^n\times\mathbb R^m
\to\mathbb R^m
$, 
find
$\y\in\mathcal Y(t,\xi_t)$ such that
\begin{equation}\label{eq:second_VI}
\left\langle
\int_0^t
R\bigl(s,\xi_s,\x,\y\bigr)\,ds
+
F\bigl(t,\xi_t,\x,\y\bigr),
\z-\y
\right\rangle
\geq 0,
\qquad
\forall \z\in\mathcal Y(t,\xi_t).
\end{equation}
Here the integral in \eqref{eq:second_VI} is understood in a pathwise sense, that
is, as the Lebesgue integral in $s$ along each realized history path of
$\bm{\xi}$.

This DSVI  \eqref{eq:closed_loop_model_projected_random} incorporates two
structural features of stochastic control and decision problems: time-dependent/nonstationary external
inputs $\bm \xi_t$ and history-dependent responses $y(t,\bm{\xi}_{\cdot\wedge t})$. Nonstationary inputs arise in a number of applications, for example,
 health monitoring, online resource allocation,  and traffic
or energy management
\cite{besbesgurzeevi15,caozhangpoor21,chen2025differential,cutlerdrusvyatskiyharchaoui23}.
In these settings, operating conditions, user behavior, data quality, and
external environments evolve over time. Here, the external uncertainty is modeled by an exogenous stochastic
process \(\bm\xi=(\bm\xi_t)_{t\in\mathcal T}\).
The dependence on the history \(\bm\xi_{\cdot\wedge t}\) models the  influence of past input-output data, longitudinal measurements, sensor histories,
cumulative exposure, delayed effects, and adaptive feedback for decision makings
\cite{mnih15, nahumshani18, depersistesi20}.
(See Figure \ref{fig:closed_loop_architecture} for a depiction of the DSVI). 

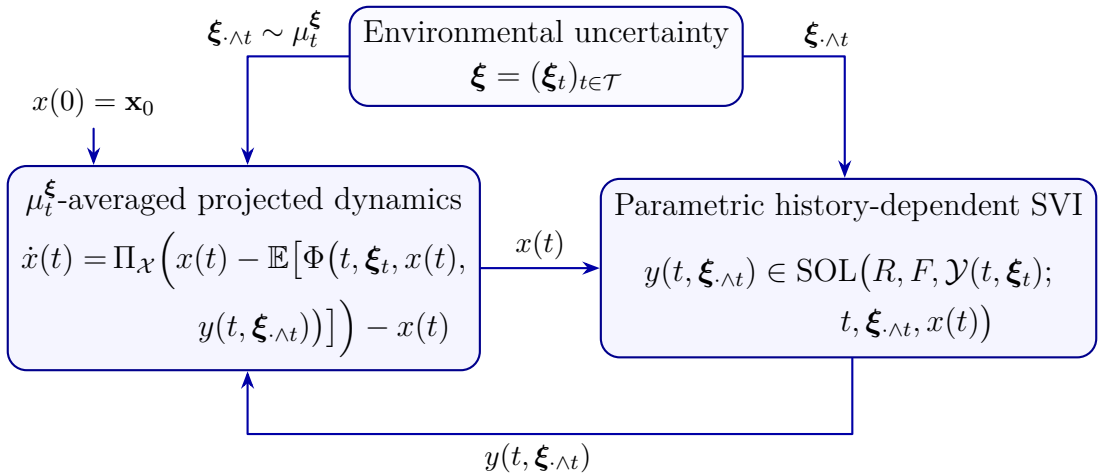
\begin{figure}[htbp]
\centering
\begin{tikzpicture}[
    scale=1,
    transform shape,
    >=Stealth,
    every path/.style={line cap=round, line join=round},
    box/.style={
        draw=blue!55!black,
        thick,
        rounded corners=2.5mm,
        fill=blue!4,
        align=center,
        font=\normalsize,
        inner sep=5pt
    },
    topbox/.style={
        draw=blue!55!black,
        thick,
        rounded corners=2.5mm,
        fill=blue!2,
        align=center,
        font=\normalsize,
        inner sep=5pt
    },
    arrow/.style={
        ->,
        thick,
        draw=blue!65!black
    },
    lab/.style={
        font=\small,
        fill=white,
        inner sep=1pt
    }
]
\node[topbox, minimum width=4cm, minimum height=1cm] (noise) at (0,2.8)
{Environmental uncertainty\\
$\bm{\xi}=(\bm{\xi}_t)_{t\in\mathcal T}$};
\node[box, minimum width=5cm, minimum height=2.5cm] (state) at (-4.0,0)
{
{\(\mu_t^{\bm{\xi}}\)-averaged projected dynamics}\\[0.9ex]
$
\begin{aligned}
\dot x(t)
=&\,
\Pi_{\mathcal X}\Bigl(
x(t)
-\mathbb E\bigl[
\Phi\bigl(
t,\bm{\xi}_t,x(t),\\
&\hspace{12mm}
y(t,\bm{\xi}_{\cdot\wedge t})
\bigr)
\bigr]
\Bigr)-x(t)
\end{aligned}
$
};

\node[box, minimum width=5cm, minimum height=2.1cm] (resp) at (4.0,0)
{
{Parametric history-dependent SVI}\\[2ex]
$
\begin{aligned}
y(t,\bm{\xi}_{\cdot\wedge t})
&\in
\operatorname{SOL}\bigl(
R,F,\mathcal Y(t,\bm{\xi}_t);\\
&\hspace{11mm}
t,\bm{\xi}_{\cdot\wedge t},x(t)
\bigr)
\end{aligned}
$
};
\coordinate (xin) at ($(state.north west)+(1.15,0)$);
\node[font=\small] (x0) at ($(xin)+(0,0.82)$) {$x(0)=\x_0$};

\draw[arrow]
  (noise.west) -- ++(-5mm,0) -|
  node[pos=0.36, above, yshift=2pt,  lab] {$\bm{\xi}_{\cdot\wedge t}\sim\mu_t^{\bm{\xi}}$}
  ($(state.north)+(0.05,0)$);

\draw[arrow]
  (noise.east) -- ++(5mm,0) -|
  node[pos=0.36, above, yshift=2pt, lab] {$\bm{\xi}_{\cdot\wedge t}$}
  ($(resp.north)+(-0.05,0)$);

\draw[arrow]
  (x0) -- (xin);

\draw[arrow]
  (state.east) --
  node[above, yshift=2pt, lab] {$x(t)$}
  (resp.west);

\draw[arrow]
  ($(resp.south)+(0.05,0)$) |- ++(0,-1) -|
  node[pos=0.26, below, yshift=-2pt, lab]
  {$y(t,\bm{\xi}_{\cdot\wedge t})$}
  ($(state.south)+(0.05,0)$);
\end{tikzpicture}
\caption{Stochastic closed-loop architecture of the history-dependent
 DSVI.}
\label{fig:closed_loop_architecture}
\end{figure}

\paragraph{\bf Well-definedness and solution to the DSVI.}
 A natural question arises regarding the well-definedness of  this general  closed-loop DSVI  model
\eqref{eq:closed_loop_model_projected_random}. Since this DSVI system includes a number of dynamic, stochastic, and variational models under suitable restrictions, such a question has been analyzed in various special cases. For instance, suppose that there exist functions
\(r,f:[0,T]\times\Xi\times\mathbb R^n\times\mathbb R^m\to\mathbb R\)
such that, for each fixed \((t,\xi,\x)\), the functions
\(r(t,\xi,\x,\cdot)\) and \(f(t,\xi,\x,\cdot)\) are differentiable and convex,
and
$
    R(t,\xi,\x,\y)=\nabla_{\y} r(t,\xi,\x,\y)$, $
    F(t,\xi,\x,\y)=\nabla_{\y} f(t,\xi,\x,\y).
$
Then,  for \(t\in\mathcal T\), \eqref{eq:closed_loop_model_projected_random_second}  is equivalent to the
parametric optimization problem
\begin{equation}\label{eq:closed_loop_model_projected_random_opt}
y(t,\bm{\xi}_{\cdot\wedge t})
\in
\operatorname*{argmin}_{\mathbf y\in\mathcal{Y}(t,\bm{\xi}_t)}
\left\{
\int_0^t
r\bigl(s,\bm{\xi}_s,x(t),\mathbf y\bigr)\,ds
+
f\bigl(t,\bm{\xi}_t,x(t),\mathbf y\bigr)
\right\}.
\end{equation}
Assume further that \(R\equiv0\). Then 
\eqref{eq:closed_loop_model_projected_random} reduces to the one-block DSVI
with parametric optimization studied in
\cite{ChenICM2026,chen2025differential} with \(\bm\xi_t\sim P_t\). If, in addition, \(P_t\equiv P\) and
\(\mathcal Y(t,\xi)\equiv\mathcal Y(\xi)\), it further reduces to the classical
DSVI with a time-independent exogenous distribution
\cite{chen2022dynamic}. If the randomness $\bm\xi$ is removed, the model becomes a
deterministic DVI, as studied for example in
\cite{camlibel2007lyapunov,chenwang13,chenwang14,pangstewart08}. If the
projected differential equation is replaced by its equilibrium condition, one
obtains a static two-stage SVI, which is related to the stochastic VI models in
\cite{chenpongwets17,chenshapirosun19,rockafellarwets17}. Finally, when \(\mathcal X=\mathbb R^n\), the deterministic case of
\eqref{eq:closed_loop_model_projected_random} includes ODEs coupled with
lower-level parametric programs \cite{landrycaboussathairer09}. The stochastic
case includes lower-level stochastic programs through their VI optimality
conditions \eqref{eq:second_VI} \cite{luochen25}.

In this paper, we will study the well-definedness, the existence and uniqueness of a solution to \eqref{eq:closed_loop_model_projected_random} under appropriate measurability, integrability, and
solvability conditions (Section~\ref{sec:well_posedness}). The key is to ensure  that the second-stage SVI
\eqref{eq:closed_loop_model_projected_random_second} admits a unique and measurable
solution and that the expected mapping is well defined.

\paragraph{\bf Stability of the DSVI.} Another mathematical question concerns the stability of the induced first-stage trajectory. Here, stability refers to
quantitative continuous dependence on the initial state and the path law of
the exogenous process.
For differential variational inequalities, dependence
of solutions on initial data and stability of solution sets under
perturbations of the problem data have been studied in
\cite{gwinner13, pangStewart09}. Related stability results for differential
variational--hemivariational systems under perturbations of initial data,
operators, and constraint sets can be found in
\cite{tangEtAl20,xiaoEtAl22}. In stochastic settings, qualitative and
quantitative stability with respect to changes in the underlying probability
distribution has been established for stochastic generalized equations
\cite{liuRomischXu14}. More recently, the statistical robustness of
sample-average-approximation solutions under data contamination has been
studied for stochastic generalized equations \cite{guoXu23}. 

In this paper, we will
provide the stability analysis based on these earlier developments \cite{
guoXu23, liuRomischXu14, pangStewart09}. For two history-dependent DSVIs
\eqref{eq:closed_loop_model_projected_random} with  different
initial states and different path laws of the exogenous processes,  we derive
an error bound for the corresponding first-stage trajectories in terms of
the differences in the initial states and the process laws. This bound
quantifies system changes propagated from
 the history-dependent second-stage response.

\paragraph{\bf Transfer learning of the DSVI in health monitoring.}

The stability  analysis of DSVI inspires us to explore the possibility of transfer learning  \cite{bendavid10,
panYang10, taylorstone09}
between  history-dependent DSVIs in the context of  health monitoring \eqref{eq:closed_loop_model_projected_random}.
In transfer learning,  a known solution  from one system (called the source task) is exploited to find approximate solutions to another ``similar'' and unknown system (called the target task).  Transfer learning   based on stability of dynamic (controlled) systems  has  been studied, for example, in  \cite{aubin1999set, CGGR26, guolixu26}, which demonstrate its efficiency especially when the relevant data for the target task are limited and the computational cost is high.

In this paper, we apply transfer learning across history-dependent DSVIs to a health-monitoring system, enabling real-time medical responses to the health states of new  users by leveraging an existing, well-trained system developed for current users with similar characteristics.
In particular, 
the first-stage equation
\eqref{eq:closed_loop_model_projected_random_first} describes the evolution of each individual's continuous health state $(x(t))_{t\ge 0}$,  which is then projected  onto
\(\mathcal X\) to yield 
health-state labels. 
The process \(\bm{\xi}\) represents external
factors affecting the individual's health evolution. For each history
\(\bm{\xi}_{\cdot\wedge t}\), the second-stage problem estimates the response
variable \(y(t,\bm{\xi}_{\cdot\wedge t})\), based on  the accumulated fitting error  \(R\) and the terminal fitting error \(F\). This response variable  is then inserted into the
feedback mapping \(\Phi\).
Theoretically, the uniqueness of the solution to the second-stage SVI allows us to identify critical features or responses for new users that are shared by existing patients in the system; 
and the stability of the induced first-stage trajectory enables the reuse of precomputed response trajectories from similar source users to  quickly identify  the health state of new individuals with limited reference data
  (for details, see
Section~\ref{sec:health_monitoring_application}).

{In summary}, this paper 
\begin{itemize}[leftmargin=17pt]
\item
 formulates a history-dependent DSVI driven by an exogenous stochastic process. The second-stage response is defined on the history of the process, and the first-stage projected state equation is governed by the induced expected-value mapping. This formulation provides a dynamic variational framework that extends a broad class of stochastic, deterministic, and optimization-constrained models;

\item
establishes well-posedness and the convergence of Sample Average Approximation (SAA) for the history-dependent DSVI.  Uniqueness, measurability, and state-Lipschitz continuity of the second-stage response are established. These properties imply existence and uniqueness of the first-stage projected state trajectory and uniform almost sure convergence of the SAA trajectory generated from independent sample paths;
\item
develops a stability analysis for the history-dependent DSVI by
deriving a local \(1/2\)-Hölder stability estimate for parametric VIs with moving
feasible sets and {obtaining} an error bound for the first-stage trajectory with
respect to the initial state and the law of the exogenous process; and 

\item applies the proposed DSVI model to an elderly health monitoring problem
 using wearable, insole, and
clinical-record features to generate history-dependent response trajectories
and health-state scores. 
\end{itemize}

Overall, the numerical results demonstrate the effectiveness of the proposed
history dependent DSVI model for health-state monitoring under both
source-domain and transfer settings. It achieves consistently high
classification performance on held-out source-domain data and maintains
comparable prediction accuracy in the target-domain setting through the
similarity-weighted transfer strategy. Meanwhile, the transfer approach
reduces online computational latency by reusing precomputed source-domain
responses while yielding stable predictions  under perturbations.

\paragraph{\bf Organization of the paper.}
The remainder of this paper is organized as follows.
Section~\ref{sec:well_posedness} studies the well-posedness and basic solution
properties of the proposed model. Section~\ref{sec:SAA_HD_DSVI} establishes the
SAA and  its convergence.
Section~\ref{sec:transfer_VI} analyzes transfer learning and stability with
respect to the initial state and the exogenous process law.
Section~\ref{sec:numerical_analysis} presents numerical experiments.
Section~\ref{sec:health_monitoring_application} applies the framework to an
elderly-health embodied intelligence system. And Section~\ref{sec:conclusion}
concludes.

\paragraph{\bf Notation.}
Throughout the paper, 
 $\mathbb R_+:=[0,\infty)$ and, for any topological space $E$, 
$\mathcal B(E)$ denotes the Borel $\sigma$-algebra on $E$.
 $\|\cdot\|$ denotes the Euclidean norm, and $\|\cdot\|_\infty$
denotes the supremum norm for functions. For any positive integer $n$,
$
\mathcal C_n:=C(\mathcal T;\mathbb R^n)
$ denotes 
the space of continuous functions from $\mathcal T$ to $\mathbb R^n$. Moreover,
$L^1(\mathcal T;\mathbb R_+)$, $AC(\mathcal T;{\mathcal X})$, and $C([0,t];\Xi)$ denote
 integrable $\mathbb R_+$-valued functions, absolutely continuous
${\mathcal X}$-valued functions, and continuous $\Xi$-valued paths, respectively; and
$L^1(\Omega;\mathbb R_+)$ is the space of $\mathbb R_+$-valued random
variables integrable with respect to $\mathbb P$.

\section{Well-Posedness and Basic Solution Properties}
\label{sec:well_posedness}

This section studies the well-posedness and basic solution properties for the
initial-value problem of the history-dependent DSVI
\eqref{eq:closed_loop_model_projected_random}. Toward this end, we impose the
following assumption.

\begin{assumption}\label{ass:existence_first_stage}
The following conditions hold.
\begin{itemize}[leftmargin=2.5em,labelsep=0.3em]
\item[\textup{(A1)}]
For each fixed $(t,\xi,\x)\in \mathbb{R}_+\times\Xi\times\mathbb{R}^n,$ the mapping
$
R(t,\xi,\x,\cdot)
$
is monotone, and for every compact ${\cal K}\subset \mathbb{R}^n$, there exists
$a_{R,{\cal K}}\in L^1(\mathcal T;\mathbb R_+)$ such that
\[
\|R(t,\xi,\x,\y)\|
\le a_{R,{\cal K}}(t),
\
\forall t\in\mathcal T,\ \xi\in\Xi,\ \x\in {\cal K},\ \y\in\overline{\mathcal Y}.
\]
There also exists $\ell_R\in L^1(\mathcal T;\mathbb R_+)$ such that, for all
$t\in\mathcal T$, $\xi_1,\xi_2\in\Xi$,
$\x_1,\x_2\in \mathbb{R}^n$, and $\y_1,\y_2\in\mathbb R^m$,
\[
\|R(t,\xi_1,\x_1,\y_1)-R(t,\xi_2,\x_2,\y_2)\|
\le
\ell_R(t)
\bigl(
\|\xi_1-\xi_2\|
+
\|\x_1-\x_2\|
+
\|\y_1-\y_2\|
\bigr).
\]

\item[\textup{(A2)}]
For the mapping
$
F:\mathbb R_+\!\times\!\Xi\times\mathbb R^n\times\mathbb R^m
\!\to\!\mathbb R^m
$, there exist $\ell_F\in L^1(\mathcal T;\mathbb R_+)$ and a constant $m_F>0$ such
that, for all $t\!\in\mathcal T$, $\xi_1,\xi_2, \xi\in\Xi$,
$\x_1,\x_2\!\in \mathbb{R}^n$, and $\y_1,\y_2\!\in\mathbb R^m$,
\[
\|F(t,\xi_1,\x_1,\y_1)-F(t,\xi_2,\x_2,\y_2)\|
\le
\ell_F(t)
\bigl(
\|\xi_1-\xi_2\|
+
\|\x_1-\x_2\|
+
\|\y_1-\y_2\|
\bigr),
\]
\[
\left\langle
F(t,\xi,\x,\y_1)-F(t,\xi,\x,\y_2),
\y_1-\y_2
\right\rangle
\ge
m_F\|\y_1-\y_2\|^2 .
\]

\item[\textup{(A3)}]
For the mapping
$
\Phi:\mathbb R_+\times\Xi\times\mathbb R^n\times\mathbb R^m
\to\mathbb R^n
$, there exists $\ell_\Phi\in L^1(\mathcal T;\mathbb R_+)$ such
that, for all $t\!\in\mathcal T$, $\xi_1,\xi_2\in\Xi$,
$\x_1,\x_2\!\in \mathbb{R}^n$, and $\y_1,\y_2\!\in\mathbb R^m$,
\[
\|\Phi(t,\xi_1,\x_1,\y_1)-\Phi(t,\xi_2,\x_2,\y_2)\|
\le
\ell_\Phi(t)
\bigl(
\|\xi_1-\xi_2\|
+
\|\x_1-\x_2\|
+
\|\y_1-\y_2\|
\bigr).
\]
Moreover, there exists a nonnegative function
$
b_\Phi:\mathcal T\times\Xi\to\mathbb R_+
$
such that
\[
\|\Phi(t,\xi,0,0)\|
\le
b_\Phi(t,\xi),
\quad
t\in\mathcal T,\ \xi\in\Xi,
\]
and
$
b_\Phi(t,\bm\xi_t)\in L^1(\Omega;\mathbb R_+)
\  \text{for a.e. } t\in\mathcal T
$
with
$
t\mapsto
\mathbb E\bigl[b_\Phi(t,\bm\xi_t)\bigr]
\in L^1(\mathcal T;\mathbb R_+).
$
\end{itemize}
\end{assumption}

Before applying the ODE existence theorem, we first present the following basic
properties of the history-dependent second-stage response \eqref{eq:closed_loop_model_projected_random_second}.

\begin{lemma}\label{lem:second_stage_properties}
Suppose that Assumption~\ref{ass:existence_first_stage}\textup{(A1)--(A2)}
holds. Then, for every $t\in\mathcal T$, every history path
$\xi_{\cdot\wedge t}$, and every $\x\in \mathbb{R}^n$, the VI
\eqref{eq:second_VI} admits a unique solution denoted by
$
\y(t,\xi_{\cdot\wedge t},\x).
$
Moreover, for fixed $t\in\mathcal T$ and $\xi_{\cdot\wedge t}$, the mapping
$\x\mapsto \y(t,\xi_{\cdot\wedge t},\x)$ is Lipschitz continuous, i.e.,
\[
\|\y(t,\xi_{\cdot\wedge t},\x_1)
-
\y(t,\xi_{\cdot\wedge t},\x_2)\|
\le
\frac{
\ell_F(t)+\int_0^t\ell_R(s)\,ds
}{
m_F
}
\|\x_1-\x_2\|.
\]
{If, in addition, the graph
\(\{(t,\omega,\y):\y\in\mathcal Y(t,\bm\xi_t(\omega))\}\)
of the set-valued mapping
\((t,\omega)\mapsto\mathcal Y(t,\bm\xi_t(\omega))\) is measurable, then
\((t,\omega,\x)\mapsto
\y(t,\bm{\xi}_{\cdot\wedge t}(\omega),\x)\)
is measurable.}
\end{lemma}

\begin{proof}
Define
$
\Psi(t,\xi_{\cdot\wedge t},\x,\y)
:=
\int_0^t
R(s,\xi_s,\x,\y)\,ds
+
F(t,\xi_t,\x,\y).
$
By Assumption~\ref{ass:existence_first_stage}\textup{(A1)}, the integral is
well defined.  Moreover, by the continuity of the mappings $R,F$ and the dominated convergence theorem \cite[Theorem~2.24]{folland1999real}, the mapping
$(\x,\y)\mapsto \Psi(t,\bm\xi_{\cdot\wedge t}(\omega),\x,\y)$ is continuous for fixed
$(t,\omega)$.

Fix $t\in\mathcal T$, $\xi_{\cdot\wedge t}$, and $\x\in \mathbb{R}^n$. For
$\y_1,\y_2\in\mathbb R^m$, by
Assumption~\ref{ass:existence_first_stage}\textup{(A1)-(A2)},
\[
\begin{aligned}
&\left\langle
\Psi(t,\xi_{\cdot\wedge t},\x,\y_1)
-
\Psi(t,\xi_{\cdot\wedge t},\x,\y_2),
\y_1-\y_2
\right\rangle                                      \\
&=
\int_0^t
\left\langle
R(s,\xi_s,\x,\y_1)-R(s,\xi_s,\x,\y_2),
\y_1-\y_2
\right\rangle ds                                  \\
&\quad+
\left\langle
F(t,\xi_t,\x,\y_1)-F(t,\xi_t,\x,\y_2),
\y_1-\y_2
\right\rangle                                      \ge
m_F\|\y_1-\y_2\|^2 .
\end{aligned}
\]
Thus the mapping $\y\mapsto\Psi(t,\xi_{\cdot\wedge t},\x,\y)$ is strongly monotone.
Since the set $\mathcal Y(t,\xi_t)$ is nonempty, compact, and convex, and the mapping
$\y\mapsto\Psi(t,\xi_{\cdot\wedge t},\x,\y)$ is continuous, the  VI
\eqref{eq:second_VI} admits a unique
solution by \cite[Theorems~2.3.3 and~2.3.5(b)]{FacchineiPang2003}. Now let us denote it by
$
\y(t,\xi_{\cdot\wedge t},\x).
$
We next prove its Lipschitz dependence on $\x$. Let
$
\y_i=\y(t,\xi_{\cdot\wedge t},\x_i)
$
for $i=1,2$. From the two VIs, we have
$$
\left\langle
\Psi(t,\xi_{\cdot\wedge t},\x_1,\y_1)
-
\Psi(t,\xi_{\cdot\wedge t},\x_2,\y_2),
\y_1-\y_2
\right\rangle
\le 0.
$$
Using the strong monotonicity of
$\y\mapsto\Psi(t,\xi_{\cdot\wedge t},\x_1,\y)$, it follows that
\[
m_F\|\y_1-\y_2\|^2
\le
\left\|
\Psi(t,\xi_{\cdot\wedge t},\x_2,\y_2)
-
\Psi(t,\xi_{\cdot\wedge t},\x_1,\y_2)
\right\|
\|\y_1-\y_2\|.
\]
By Assumption~\ref{ass:existence_first_stage}\textup{(A1)-(A2)},
\[
\left\|
\Psi(t,\xi_{\cdot\wedge t},\x_2,\y_2)
-
\Psi(t,\xi_{\cdot\wedge t},\x_1,\y_2)
\right\|
\le
\left(
\int_0^t\ell_R(s)\,ds+\ell_F(t)
\right)
\|\x_1-\x_2\|.
\]
Dividing by $m_F\|\y_1-\y_2\|$ gives
\[
\|\y(t,\xi_{\cdot\wedge t},\x_1)
-
\y(t,\xi_{\cdot\wedge t},\x_2)\|
\le
\frac{
\ell_F(t)+\int_0^t\ell_R(s)\,ds
}{
m_F
}
\|\x_1-\x_2\|.
\]

Finally, we prove the measurability of the mapping
$
(t,\omega,\x)\mapsto
\y\bigl(t,\bm{\xi}_{\cdot\wedge t}(\omega),\x\bigr).
$
Define the gap function
$
g(t,\omega,\x,\y)
:=
\sup_{\z\in\mathcal Y(t,\bm\xi_t(\omega))}
\left\langle
\Psi\bigl(t,\bm\xi_{\cdot\wedge t}(\omega),\x,\y\bigr),
\y-\z
\right\rangle .
$
By the progressive measurability of $\bm\xi$ and the continuity of $\Psi$,
the function inside the supremum in the definition of $g$ is measurable. Since the graph of
$
(t,\omega)\mapsto\mathcal Y(t,\bm\xi_t(\omega))
$
is measurable, \cite[Theorem~8.2.11]{aubin1999set} implies that $g$ is
measurable.
Moreover, for each fixed $(t,\omega,\x)$, the gap function satisfies
$
g(t,\omega,\x,\y)\ge 0
$
for all
$
\y\in\mathcal Y(t,\bm\xi_t(\omega))
$,
and
$
g(t,\omega,\x,\y)=0
$
if and only if $\y$ solves the VI \eqref{eq:second_VI}. Note that the VI
\eqref{eq:second_VI} has a unique solution. Therefore
$
\y\bigl(t,\bm\xi_{\cdot\wedge t}(\omega),\x\bigr)
=
\argmin_{\y\in\mathcal Y(t,\bm\xi_t(\omega))}
g(t,\omega,\x,\y),
$
and
\((t,\omega,\x)\mapsto
\y(t,\bm\xi_{\cdot\wedge t}(\omega),\x)\)
is measurable.
\end{proof}

We are now ready to establish existence of solutions to the model (\ref{eq:closed_loop_model_projected_random}).

\begin{theorem}\label{thm:first_stage_existence}
Suppose that Assumption~\ref{ass:existence_first_stage} holds and that the graph
of the multifunction
$(t,\omega)\mapsto\mathcal Y(t,\bm\xi_t(\omega))$
is measurable. For $t\in\mathcal T$, define
$\kappa_Y(t):=\bigl(\ell_F(t)+\int_0^t\ell_R(s)\,ds\bigr)/m_F$
and $\beta(t):=\ell_\Phi(t)\kappa_Y(t)$.
Assume further that
$\beta\in L^1(\mathcal T;\mathbb R_+)$.
Then, for every $\x_0\in\mathcal X$, system
\eqref{eq:closed_loop_model_projected_random} admits a unique solution
$(x,y)$, where $x\in AC(\mathcal T;\mathcal X)$ and
$y(t,\bm{\xi}_{\cdot\wedge t})
=\y\bigl(t,\bm{\xi}_{\cdot\wedge t},x(t)\bigr)$.
\end{theorem}

\begin{proof}
For a fixed $\x\in \mathbb{R}^n$, define
$
\hat y_{\x}(t,\omega)
:=
\y\bigl(t,\bm{\xi}_{\cdot\wedge t}(\omega),\x\bigr).
$
By Lemma~\ref{lem:second_stage_properties}, the mapping
$(t,\omega)\mapsto \hat y_{\x}(t,\omega)$
is measurable for each fixed $\x\in \mathbb{R}^n$. Moreover, since
$
\hat y_{\x}(t,\omega)
\in
\mathcal Y(t,\bm{\xi}_t(\omega))
\subseteq \overline{\mathcal Y},
$
we have
$
\|\hat y_{\x}(t,\omega)\|\le M_Y
= \sup_{\y\in\overline{\mathcal Y}}\|\y\|<\infty$, due to the compactness of 
$\overline{\mathcal Y}$.

We first verify that the expected-value mapping is well defined. By the
measurability of $\hat y_{\x}$, the progressive measurability of $\bm{\xi}$,
and the continuity of $\Phi$, the mapping
$
(t,\omega)
\mapsto
\Phi\bigl(t,\bm{\xi}_t(\omega),\x,\hat y_{\x}(t,\omega)\bigr)
$
is measurable.

Furthermore, by Assumption~\ref{ass:existence_first_stage}\textup{(A3)}, for
$t\in\mathcal T$ a.e. and every $\x\in \mathbb{R}^n$,
\[
\begin{aligned}
\left\|
\Phi\bigl(t,\bm{\xi}_t(\omega),\x,\hat y_{\x}(t,\omega)\bigr)
\right\|
&\le
\left\|
\Phi\bigl(t,\bm{\xi}_t(\omega),0,0\bigr)
\right\|
+
\ell_\Phi(t)
\bigl(
\|\x\|+\|\hat y_{\x}(t,\omega)\|
\bigr)\\
&\le
b_\Phi\bigl(t,\bm{\xi}_t(\omega)\bigr)
+
\ell_\Phi(t)
\bigl(
\|\x\|+M_Y
\bigr).
\end{aligned}
\]
Since
$
b_\Phi(t,\bm{\xi}_t)\in L^1(\Omega;\mathbb R_+)
$
for $t\in\mathcal T$ a.e., the integrand is $\mathbb P$-integrable for
$t$ a.e., hence the well-definedness of  the expected-value mapping
$
\bar\Phi(t,\x)
\!:=\!
\mathbb E\left[
\Phi\bigl(t,\bm{\xi}_t(\cdot),\x,\hat y_{\x}(t,\cdot)\bigr)
\right]
$.

Next, we show that $\bar\Phi$ is measurable in $t$ and Lipschitz continuous
in $\x$ with an integrable Lipschitz modulus. For each fixed $\x\in \mathbb{R}^n$, since
$
(t,\omega)\mapsto
\Phi\bigl(t,\bm{\xi}_t(\omega),\x,\hat y_{\x}(t,\omega)\bigr)
$
is measurable and $\mathbb P$-integrable for $t\in\mathcal T$ a.e., the
mapping
$
t\mapsto \bar\Phi(t,\x)
$
is measurable.
For $\x_1,\x_2\in \mathbb R^n$, by Assumption~\ref{ass:existence_first_stage}
\textup{(A3)} and Lemma~\ref{lem:second_stage_properties}, for a.e.
$t\in\mathcal T$,
\begin{align}
&\|\bar\Phi(t,\x_1)-\bar\Phi(t,\x_2)\|
\le
\mathbb E\left[
\left\|
\Phi\bigl(t,\bm{\xi}_t(\cdot),\x_1,\hat y_{\x_1}(t,\cdot)\bigr)
-
\Phi\bigl(t,\bm{\xi}_t(\cdot),\x_2,\hat y_{\x_2}(t,\cdot)\bigr)
\right\|
\right] \nonumber \\
&\le
\ell_\Phi(t)
\left(
\|\x_1-\x_2\|
+
\mathbb E\left[
\|\hat y_{\x_1}(t,\cdot)-\hat y_{\x_2}(t,\cdot)\|
\right]
\right)\le
\ell_\Phi(t)\bigl(1+\kappa_Y(t)\bigr)
\|\x_1-\x_2\|.
\label{bar_phi_estimate}
\end{align}
Since
$
\ell_\Phi$ and $\beta\in L^1(\mathcal T;\mathbb R_+),
$
the mapping $\bar\Phi(t,\cdot)$ is Lipschitz continuous with an integrable
Lipschitz modulus.
Define, for $(t,\x)\in\mathcal T\times\mathbb R^n$,
\begin{align}
H(t,\x)
:=
\Pi_{\mathcal X}\bigl(\x-\bar\Phi(t,\x)\bigr)-\x .\label{eq:H_t_x}
\end{align}
Since the Euclidean projection onto a nonempty closed convex set is
nonexpansive, for a.e. $t\in\mathcal T$ and all
$\x_1,\x_2\in\mathbb R^n$, it follows that
{\begin{align}
&\|H(t,\x_1)\!-\!H(t,\x_2)\|
\!\le\!
\left\|
\Pi_{\mathcal X}\bigl(\x_1\!-\!\bar\Phi(t,\x_1)\bigr)
\!-\!
\Pi_{\mathcal X}\bigl(\x_2\!-\!\bar\Phi(t,\x_2)\bigr)
\right\|
+
\|\x_1-\x_2\| \notag \\
&\le
2\|\x_1-\x_2\|
+
\|\bar\Phi(t,\x_1)-\bar\Phi(t,\x_2)\|
\le
\bigl(2+\ell_\Phi(t)(1+\kappa_Y(t))\bigr)
\|\x_1-\x_2\|.
\label{eq:linear_growth_H_t_x}
\end{align}}
Hence $H(t,\cdot)$ is Lipschitz continuous with an integrable Lipschitz
modulus.

Then, since $\x_0\in {\mathcal X}$, for $\x\in\mathbb R^n$ we have
\[
\begin{aligned}
\|H(t,\x)\|
&=
\left\|
\Pi_{\mathcal X}\bigl(\x-\bar\Phi(t,\x)\bigr)-\x
\right\|                                                   \le
\left\|
\Pi_{\mathcal X}\bigl(\x-\bar\Phi(t,\x)\bigr)-\Pi_{\mathcal X}(\x)
\right\|
+
\left\|
\Pi_{\mathcal X}(\x)-\x
\right\|                                                   \\
&\le
\|\bar\Phi(t,\x)\|
+
\operatorname{dist}(\x,{\mathcal X})                                  \le
\mathbb E\bigl[b_\Phi(t,\bm\xi_t)\bigr]
+
\ell_\Phi(t)(\|\x\|+M_Y)
+
\|\x-\x_0\|                                                \\
&\le
\bigl(1+\ell_\Phi(t)\bigr)\|\x\|
+
\mathbb E\bigl[b_\Phi(t,\bm\xi_t)\bigr]
+
\ell_\Phi(t)M_Y
+
\|\x_0\|.
\end{aligned}
\]
Noting that
$
t\mapsto\mathbb E[b_\Phi(t,\bm\xi_t)]
$
and $\ell_\Phi$ belong to $L^1(\mathcal T;\mathbb R_+)$, this gives a
linear growth bound of \eqref{eq:H_t_x} in $\x$ with integrable coefficients.
Therefore, by the classical existence and uniqueness theorem for
Carath\'eodory ordinary differential equations; see
\cite[Theorem~2.17 and the following remark]{teschl2012ordinary}, the initial
value problem
\begin{align}
\dot x(t)=H(t,x(t)),
\quad
x(0)=\x_0,\label{Expected_system}
\end{align}
admits a unique absolutely continuous solution on $\mathcal T$, still denoted
by $x$.

It remains to show that $x(t)\in {\mathcal X}$ for all $t\in\mathcal T$. Let
$
z(t):=\Pi_{\mathcal X}\bigl(x(t)-\bar\Phi(t,x(t))\bigr).
$
Then $z(t)\in {\mathcal X}$ for a.e. $t$, and
$
\dot x(t)+x(t)=z(t).
$
Hence
$
x(t)
=
e^{-t}\x_0
+
\int_0^t e^{-(t-s)}z(s)\,ds .
$
Since $\x_0\in {\mathcal X}$, $z(s)\in {\mathcal X}$ for a.e. $s$, and
$
e^{-t}+\int_0^t e^{-(t-s)}\,ds=1,
$
the vector $x(t)$ is the limit of finite convex combinations of points in
${\mathcal X}$. By the closedness and convexity of ${\mathcal X}$, this limit also belongs to ${\mathcal X}$.
Therefore $x(t)\in {\mathcal X}$ for all $t\in\mathcal T$.

Finally, by Lemma~\ref{lem:second_stage_properties},
$
y(t,\bm{\xi}_{\cdot\wedge t})
:=
\y\bigl(t,\bm{\xi}_{\cdot\wedge t},x(t)\bigr)
$ satisfies the
second-stage SVI  \eqref{eq:closed_loop_model_projected_random_second}. Moreover, by
the definition of $\bar\Phi$, the state trajectory $x$ satisfies the projected
expected-value state equation. Hence $(x,y)$ satisfies
\eqref{eq:closed_loop_model_projected_random}.
\end{proof}

\section{Sample Average Approximation of the\\ History-Dependent DSVI}
\label{sec:SAA_HD_DSVI}
We next consider a sample average approximation (SAA) of the expectation
appearing in the first-stage dynamics
\eqref{eq:closed_loop_model_projected_random_first}. We first
rewrite the expectation as an integral over the history paths of the
exogenous process, as suggested in the Introduction.

To this end, first define 
$\mathfrak{M}_t(\Xi):=\{\gamma:[0,t]\to\Xi \mid \gamma
\text{ is measurable}\}$  as the space of all measurable paths from $[0,t]$ into the Borel set $\Xi$,   equipped with an appropriate \(\sigma\)-algebra \(\mathcal M_t\).
Such \(\mathcal M_t\) ensures that endpoint map \(\gamma\mapsto\gamma(t)\) is
\(\mathcal M_t/\mathcal B(\Xi)\)-measurable in the sense that
\(\{\gamma\in\mathfrak M_t(\Xi):\gamma(t)\in B\}\in\mathcal M_t\) for every
$B$ in the  Borel \(\sigma\)-algebra on  \(\Xi\) denoted by $\mathcal B(\Xi)$, and the   history process
\(\bm{\xi}_{\cdot\wedge t}:\Omega\to\mathfrak M_t(\Xi)\) is
\(\mathcal F_t/\mathcal M_t\)-measurable in the sense that
\(\{\omega\in\Omega:\bm{\xi}_{\cdot\wedge t}(\omega)\in C\}\in\mathcal F_t\)
for every \(C\in\mathcal M_t\).  
For each $\gamma\in\mathfrak M_t(\Xi)$ and $\x\in\mathcal X$, let
$\y(t,\gamma,\x)$ denote the unique solution of
\eqref{eq:second_VI}, and define
$G(t,\x,\gamma):=
\Phi\bigl(t,\gamma(t),\x,\y(t,\gamma,\x)\bigr)$.
Let
$\mu_t^{\bm\xi}:=
\mathbb P\circ(\bm\xi_{\cdot\wedge t})^{-1}$
be the law of $\bm\xi_{\cdot\wedge t}$ on
$(\mathfrak M_t(\Xi),\mathcal M_t)$.

Now, suppose that, for
every $(\x,\y)\in\mathcal X\times\mathbb R^m$, the map
$\gamma\mapsto\int_0^t R(s,\gamma(s),\x,\y)\,ds$ is
$\mathcal M_t/\mathcal B(\mathbb R^m)$-measurable. Then, by the same
measurable-selection argument as in
Lemma~\ref{lem:second_stage_properties}, the second-stage solution map
$(\gamma,\x)\mapsto\y(t,\gamma,\x)$ is
$\mathcal M_t\otimes\mathcal B(\mathcal X)/
\mathcal B(\mathbb R^m)$-measurable. Therefore, the map
$(\gamma,\x)\mapsto G(t,\x,\gamma)$ is
$\mathcal M_t\otimes\mathcal B(\mathcal X)/
\mathcal B(\mathbb R^n)$-measurable. If, in addition,
$\gamma\mapsto G(t,\x,\gamma)$ is $\mu_t^{\bm\xi}$-integrable, then the
change-of-variables formula for image measures
\cite[Theorem~3.6.1]{Bogachev2007MeasureTheory} yields
\begin{equation}\label{bar_Phi_def}
\begin{aligned}
\bar\Phi(t,\x)
:=
\mathbb E\Bigl[
\Phi\bigl(
t,\bm\xi_t,\x,
\y(t,\bm\xi_{\cdot\wedge t},\x)
\bigr)
\Bigr]=
\int_{\mathfrak M_t(\Xi)}
G(t,\x,\gamma)
\,\mu_t^{\bm\xi}(d\gamma).
\end{aligned}
\end{equation}

Since the integral in \eqref{bar_Phi_def} is generally unavailable in
applications, we approximate $\bar\Phi$ by an empirical average based on
independent realizations of the exogenous process, while the second-stage
response is still computed pathwise through
\eqref{eq:closed_loop_model_projected_random_second}.

Let \(\{\bm\xi^i\}_{i\ge1}\) be i.i.d. copies of \(\bm\xi\). For each
\(N\ge1\), the SAA of \(\bar\Phi\) is defined by
\begin{equation}
    \label{eq:SAA_of_Phi}
\bar\Phi_N(t,\x)
:=
\frac1N
\sum_{i=1}^N
G\bigl(t,\x,\bm\xi^i_{\cdot\wedge t}\bigr)
=
\frac1N
\sum_{i=1}^N
\Phi\bigl(
t,\bm\xi_t^i,\x,
\y(t,\bm\xi^i_{\cdot\wedge t},\x)
\bigr).
\end{equation}

Let $x$ denote the solution of the expected-value ODE
\eqref{Expected_system}, its SAA counterpart is obtained by replacing
$\bar\Phi$ with $\bar\Phi_N$ such that
\begin{equation}\label{eq:SAA_state_equation}
\left\{
\begin{aligned}
\dot x_N(t)
&=
\Pi_{\mathcal X}\Bigl(
x_N(t)
-
\bar\Phi_N\bigl(t,x_N(t)\bigr)
\Bigr)
-
x_N(t),
\quad t\in\mathcal T,
\\
x_N(0)&=x_0.
\end{aligned}
\right.
\end{equation}
Here, \(x_N\) denotes the solution of the SAA system
\eqref{eq:SAA_state_equation}.

{

\begin{remark}

The above measurable history-space formulation is compatible with standard canonical path spaces. If the exogenous process has continuous sample paths and $\Xi$ is closed, one may take $\mathfrak M_t(\Xi)=C([0,t];\Xi)$ and $\mathcal M_t=\mathcal B(C([0,t];\Xi))$, where $C([0,t];\Xi)$ is equipped with the uniform norm. Since $(s,\gamma)\mapsto\gamma(s)$ and $R$ are continuous, $(s,\gamma)\mapsto R(s,\gamma(s),\x,\y)$ is jointly Borel measurable. Together with the integrable bound in Assumption~\ref{ass:existence_first_stage}\textup{(A1)}, this implies that $\gamma\mapsto\int_0^t R(s,\gamma(s),\x,\y)\,ds$ is Borel measurable.

If the exogenous process has c\`adl\`ag sample paths, one may instead take
$\mathfrak M_t(\Xi)=D([0,t];\Xi)$ and
$\mathcal M_t=\mathcal B(D([0,t];\Xi))$, where $D([0,t];\Xi)$ is endowed
with the Skorokhod $J_1$ topology. In this case, the evaluation map
$(s,\gamma)\mapsto\gamma(s)$ is jointly Borel measurable. Hence, by the
continuity of $R$ and Assumption~\ref{ass:existence_first_stage}\textup{(A1)},
the same conclusion follows. For standard facts on the Skorokhod space, see
\cite[Chapter~3]{billingsley1999convergence}.
\end{remark}
}

For the SAA convergence analysis, we assume that \(\mathcal X\) is compact. We first prove the convergence of
the SAA mapping \(\bar\Phi_N\) to \(\bar\Phi\), and then use this approximation
result to establish the convergence of the SAA state trajectory \(x_N\) to the
original trajectory \(x\). Since \(\mathcal X\) is compact, set
\(M_X:=\sup_{\x\in\mathcal X}\|\x\|<\infty\).

\begin{lemma}\label{lem:SAA_drift_as_convergence}
Suppose that the conditions of
Theorem~\ref{thm:first_stage_existence} hold and that \(\mathcal X\) is
compact. Then
$
\int_0^T \sup_{\x\in\mathcal X}
\left\|
\bar\Phi_N(t,\x)
-
\bar\Phi(t,\x)
\right\|\,dt
\rightarrow 0
\  \text{\rm w.p.1\ as}\  N\to\infty .
$
\end{lemma}

\begin{proof}
Fix $t\in\mathcal T$. Note that
$\gamma\mapsto G(t,\x,\gamma)$ is measurable for each $\x\in {\mathcal X}$. Moreover,
by Lemma~\ref{lem:second_stage_properties} and
Assumption~\ref{ass:existence_first_stage}\textup{(A3)}, for all
$\x_1,\x_2\in {\mathcal X}$,
\[
\begin{aligned}
\|G(t,\x_1,\gamma)-G(t,\x_2,\gamma)\|
&\le
\ell_\Phi(t)
\bigl(
\|\x_1-\x_2\|
+
\|\y(t,\gamma,\x_1)-\y(t,\gamma,\x_2)\|
\bigr)  \\
&\le
\ell_\Phi(t)(1+\kappa_Y(t))\|\x_1-\x_2\|.
\end{aligned}
\]
Set
$
L(t):=\ell_\Phi(t)(1+\kappa_Y(t)).
$
Then $L\in L^1(\mathcal T;\mathbb R_+)$ by assumption.
Since
$
\y(t,\gamma,\x)\in\mathcal Y(t,\gamma(t))\subseteq\overline{\mathcal Y},
$
Assumption~\ref{ass:existence_first_stage}\textup{(A3)} gives
$
\sup_{\x\in {\mathcal X}}\|G(t,\x,\gamma)\|
\le
J(t,\gamma)
$
with
\[
J(t,\gamma)
:=
b_\Phi(t,\gamma(t))
+
\ell_\Phi(t)(M_X+M_Y).
\]
For a.e. $t\in\mathcal T$, it then follows that
\[
\int_{\mathfrak M_t(\Xi)}
J(t,\gamma)\,\mu_t^{\bm\xi}(d\gamma)
=
\mathbb E\bigl[b_\Phi(t,\bm\xi_t)\bigr]
+
\ell_\Phi(t)(M_X+M_Y)
<\infty .
\]

We now verify the bracketing condition in
\cite[Definition~2.1.6 and Theorem~2.4.1]{vanderVaartWellner1996}
for each component of $G$. Let $G_r$ be the $r$-th component of $G$,
$r=1,\ldots,n$, and define
$
\mathcal F_{t,r}
:=
\{\,\gamma\mapsto G_r(t,\x,\gamma):\x\in {\mathcal X}\,\}.
$
For any $\varepsilon>0$, the compactness of ${\mathcal X}$ gives points
$\x_1,\ldots,\x_M\in {\mathcal X}$ such that every $\x\in {\mathcal X}$ satisfies
$
\|\x-\x_j\|\le \varepsilon/(2L(t))
$
for some $j$, with the convention that the claim is trivial if $L(t)=0$.
Hence
$
|G_r(t,\x,\gamma)-G_r(t,\x_j,\gamma)|
\le
\frac{\varepsilon}{2}$ for
$\gamma\in\mathfrak M_t(\Xi).
$
Therefore $G_r(t,\x,\cdot)$ is contained in the bracket
$
\left[
G_r(t,\x_j,\cdot)-\frac{\varepsilon}{2},
\,
G_r(t,\x_j,\cdot)+\frac{\varepsilon}{2}
\right],
$
whose $L^1(\mu_t^{\bm\xi})$-width is equal to $\varepsilon$. Here the \(L^1(\mu_t^{\bm\xi})\)-width of a bracket \([l,u]\) means
$
\|u-l\|_{L^1(\mu_t^{\bm\xi})}
:=
\int_{\mathfrak M_t(\Xi)}
|u(\gamma)-l(\gamma)|\,\mu_t^{\bm\xi}(d\gamma).
$ 
Thus
$\mathcal F_{t,r}$ admits finitely many brackets of arbitrary positive width.
By the Glivenko--Cantelli theorem
\cite[Theorem~2.4.1]{vanderVaartWellner1996}, applied on
$(\mathfrak M_t(\Xi),\mathcal M_t,\mu_t^{\bm\xi})$, we obtain
\[
\sup_{\x\in {\mathcal X}}
\left|
\frac1N\sum_{i=1}^N
G_r(t,\x,\bm\xi^i_{\cdot\wedge t})
-
\mathbb E\bigl[
G_r(t,\x,\bm\xi_{\cdot\wedge t})
\bigr]
\right|
\longrightarrow 0
\quad \text{\rm w.p.1\ as}\  N\to\infty
\]
for each $r=1,\ldots,n$ and for a.e. $t\in\mathcal T$. Since $n$ is finite,
this implies that
\[
\sup_{\x\in\mathcal X}
\left\|
\bar\Phi_N(t,\x)
-
\bar\Phi(t,\x)
\right\|\to 0
\  \text{\rm w.p.1\ as}\ N\to\infty
\]
for a.e. $t\in\mathcal T$.
For \(t\in\mathcal T\), define
\begin{align}
D_N(t)
:=
\sup_{\x\in\mathcal X}
\left\|
\bar\Phi_N(t,\x)
-
\bar\Phi(t,\x)
\right\|.\label{eq:D_N_t}
\end{align}
To prove Lemma~\ref{lem:SAA_drift_as_convergence}, it suffices to show that for
any $\varepsilon>0$ and $\text{a.s. }$ $\omega\in\Omega$, there exists
$N^*=N^*(\varepsilon,\omega)$ such that  for all
$N\ge N^*$, 
$
\int_0^T D_N(t,\omega)\,dt
\le \varepsilon.
$
Here $D_N(t)=D_N(t,\omega)$ is a random variable depending on $\omega$, and we  next prove its almost sure convergence.

Since ${\mathcal X}$ is compact and
$\x\mapsto \bar\Phi_N(t,\x)-\bar\Phi(t,\x)$ is continuous, there exists
a countable dense subset $\{\x_j\}_{j\ge1}\subset {\mathcal X}$ such that
$
D_N(t)
=
\sup_{j\ge1}
\|
\bar\Phi_N(t,\x_j)
-
\bar\Phi(t,\x_j)
\|.
$
Hence $(t,\omega)\mapsto D_N(t,\omega)$ is measurable and the set
$
E:=
\{(t,\omega)\in\mathcal T\times\Omega:
D_N(t,\omega)\not\to0\}
$
belongs to $\mathcal B(\mathcal T)\otimes\mathcal F$.
Let $\lambda$ denote the Lebesgue measure on $\mathcal T$. By
\cite[Theorem~2.36]{folland1999real}, we have
$
(\lambda\times\mathbb P)(E)
=
\int_0^T \mathbb P(E_t)\,dt
=
\int_\Omega \lambda(E^\omega)\,d\mathbb P(\omega),
$
where
$
E_t:=\{\omega\in\Omega:D_N(t,\omega)\not\to0\},
$ and $
E^\omega:=\{t\in\mathcal T:D_N(t,\omega)\not\to0\}.
$
Since $\mathbb P(E_t)=0$ for every fixed $t\in\mathcal T$, it follows that
$
(\lambda\times\mathbb P)(E)=0.
$
Consequently,
$
\lambda(E^\omega)=0
\ \text{for }\text{a.s. } \omega .
$
Thus there exists $\Omega_0\in\mathcal F$ with $\mathbb P(\Omega_0)=1$ such
that, for every $\omega\in\Omega_0$,
$
D_N(t,\omega)\to0
\  \text{as } N\to\infty
\  \text{for a.e. } t\in\mathcal T .
$
Then,
we prove the convergence of $
\int_0^T D_N(t)\,dt.
$ For $K>0$, define
$
G^K(t,\x,\gamma)
:=
G(t,\x,\gamma)\mathbf 1_{\{J(t,\gamma)\le K\}},
$
and
\[
D_N^K(t)
:=
\sup_{\x\in {\mathcal X}}
\left\|
\frac1N\sum_{i=1}^N
G^K(t,\x,\bm\xi^i_{\cdot\wedge t})
-
\mathbb E\bigl[
G^K(t,\x,\bm\xi_{\cdot\wedge t})
\bigr]
\right\|.
\]
For fixed $K$, applying the preceding argument to $G^K$ gives
$
D_N^K(t)\to0
$
$\text{\rm w.p.1\ as}\  N\to\infty$  for a.e. $t$, and
$
0\le D_N^K(t)\le 2K.
$
Hence, by the dominated convergence theorem \cite[Theorem~2.24]{folland1999real},
$
\int_0^T D_N^K(t)\,dt
\to 0\
\text{\rm w.p.1\ as}\  N\to\infty.
$
Moreover, for every $K>0$,
\[
D_N(t)
\le
D_N^K(t)
+
\frac1N\sum_{i=1}^N
J(t,\bm\xi^i_{\cdot\wedge t})
\mathbf 1_{\{J(t,\bm\xi^i_{\cdot\wedge t})>K\}}
+
\mathbb E\left[
J(t,\bm\xi_{\cdot\wedge t})
\mathbf 1_{\{J(t,\bm\xi_{\cdot\wedge t})>K\}}
\right].
\]
Integrating over $\mathcal T$ yields
\[
\begin{aligned}
\int_0^T D_N(t)\,dt
&\le
\int_0^T D_N^K(t)\,dt+
\frac1N\sum_{i=1}^N
\int_0^T
J(t,\bm\xi^i_{\cdot\wedge t})
\mathbf 1_{\{J(t,\bm\xi^i_{\cdot\wedge t})>K\}}\,dt\\
&\quad +
\int_0^T
\mathbb E\left[
J(t,\bm\xi_{\cdot\wedge t})
\mathbf 1_{\{J(t,\bm\xi_{\cdot\wedge t})>K\}}
\right]dt .
\end{aligned}
\]
By the strong law of large numbers, as $N\to \infty$,
\[
\frac1N\sum_{i=1}^N
\int_0^T
J(t,\bm\xi^i_{\cdot\wedge t})
\mathbf 1_{\{J(t,\bm\xi^i_{\cdot\wedge t})>K\}}\,dt
\to
\int_0^T
\mathbb E\left[
J(t,\bm\xi_{\cdot\wedge t})
\mathbf 1_{\{J(t,\bm\xi_{\cdot\wedge t})>K\}}
\right]dt\ \text{\rm w.p.1}.
\]
Therefore,
$
\limsup_{N\to\infty}
\int_0^T D_N(t)\,dt
\le
2
\int_0^T
\mathbb E\left[
J(t,\bm\xi_{\cdot\wedge t})
\mathbf 1_{\{J(t,\bm\xi_{\cdot\wedge t})>K\}}
\right]dt
\ \text{\rm a.s.}
$
Finally, since
\[
\int_0^T
\mathbb E\bigl[
J(t,\bm\xi_{\cdot\wedge t})
\bigr]dt
<\infty,
\]
the dominated convergence theorem \cite[Theorem~2.24]{folland1999real} gives
$
\int_0^T
\mathbb E\left[
J(t,\bm\xi_{\cdot\wedge t})
\mathbf 1_{\{J(t,\bm\xi_{\cdot\wedge t})>K\}}
\right]dt
\to 0$ as $K\to\infty.
$
\end{proof}

\begin{theorem}\label{thm:SAA_convergence}
Suppose that the assumptions of Theorem~\ref{thm:first_stage_existence} hold.
Then, for each $N$, the SAA ODE
\eqref{eq:SAA_state_equation} admits a unique solution
$
x_N\in AC(\mathcal T;{\mathcal X}).
$
Moreover,
$$
\sup_{t\in\mathcal T}
\|x_N(t)-x(t)\|
\rightarrow 0
\ \text{\rm w.p.1\ as}\  N\to\infty.
$$
\end{theorem}

\begin{proof}
For each $N$, define
$
H_N(t,\x)
:=
\Pi_{\mathcal X}\bigl(
\x-\bar\Phi_N(t,\x)
\bigr)-\x .
$
By the same argument as in Theorem~\ref{thm:first_stage_existence},
\eqref{eq:SAA_state_equation} admits a unique solution
$
x_N\in AC(\mathcal T;{\mathcal X}).
$

{Recalling \eqref{eq:H_t_x} and
since $\Pi_{\mathcal X}$ is nonexpansive, we have
$
\|H_N(t,\x)-H(t,\x)\|
\le
\|\bar\Phi_N(t,\x)-\bar\Phi(t,\x)\|.
$}
Therefore, by \eqref{eq:linear_growth_H_t_x} and noting \eqref{eq:D_N_t}, we have for any $t\in\mathcal T$,
\[
\begin{aligned}
\|x_N(t)-x(t)\|
&\le
\int_0^t
\left\|
H_N(s,x_N(s))
-
H(s,x(s))
\right\|\,ds                                      \\
&\le
\int_0^T D_N(s)\,ds
+
\int_0^t
\bigl(2+\ell_\Phi(s)(1+\kappa_Y(s))\bigr)
\|x_N(s)-x(s)\|\,ds .
\end{aligned}
\]
By Gronwall's inequality \cite[Lemma~2.7]{teschl2012ordinary}, it follows that
$
\sup_{t\in\mathcal T}
\|x_N(t)-x(t)\|
\le
C_T
\int_0^T D_N(s)\,ds,
$
where
$
C_T
:=
\exp\left(
\int_0^T
\bigl(2+\ell_\Phi(s)(1+\kappa_Y(s))\bigr)\,ds
\right).
$
Therefore, the assertion follows from Lemma~\ref{lem:SAA_drift_as_convergence}.
\end{proof}

\begin{remark}[Time discretization]
Theorem~\ref{thm:SAA_convergence} concerns the convergence of the
continuous-time SAA trajectory. For numerical implementation, the projected ODE
\eqref{Expected_system} and its SAA counterpart
\eqref{eq:SAA_state_equation} can be further approximated by suitable
time-discretization schemes. Related convergence results for discrete-time
approximations of differential inclusions and dynamic variational inequality
systems can be found in
\cite{chen2025differential,chen2022dynamic}. Since the present formulation is
obtained by substituting the second-stage response map into the first-stage
dynamics, the same discretization arguments apply to the resulting ODE and yield convergence of the discrete trajectories as the mesh
size tends to zero, provided that the expected-value mappings are
well defined.
\end{remark}

\section{Transfer Learning in the History-Dependent DSVI Model}
\label{sec:transfer_VI}

In this section, we study transfer learning in the history-dependent DSVI model
\eqref{eq:closed_loop_model_projected_random} through a source-to-target
stability analysis. The question is how the first-stage state trajectory changes
when the initial state and the law of the exogenous process are changed from a
source environment to a related target environment.

Assume throughout this section that the exogenous processes have continuous
sample paths. Let
$
    \Lambda
    :=
    \{\gamma\in\mathcal C_\ell:\gamma(t)\in\Xi,\ \forall t\in\mathcal T\}
$
be the path space. Let
\(\bm{\xi}^\alpha=(\bm{\xi}_t^\alpha)_{t\in\mathcal T}\) and
\(\bm{\xi}^\beta=(\bm{\xi}_t^\beta)_{t\in\mathcal T}\) be two
\(\Xi\)-valued exogenous stochastic processes, and let
\(\x_0^\alpha,\x_0^\beta\in\mathcal X\) be two initial states. For
\(\rho\in\{\alpha,\beta\}\), define
$
    \mu^\rho
    :=
    \mathbb P\circ(\bm\xi^\rho)^{-1},
$
the law of the full sample path on \(\Lambda\).

Since the second-stage SVI \eqref{eq:closed_loop_model_projected_random_second} admits a unique solution, for
\(t\in\mathcal T\), \(\gamma\in\Lambda\), and \(\x\in\mathcal X\), we denote
this solution by
\(\y(t,\gamma_{\cdot\wedge t},\x)\). For each
\(\rho\in\{\alpha,\beta\}\), define
\begin{equation}\label{bar_phi_rho}
\bar\Phi^\rho(t,\x)
:=
\int_{\Lambda}
\Phi\bigl(
t,\gamma(t),\x,
\y(t,\gamma_{\cdot\wedge t},\x)
\bigr)
\,\mu^\rho(d\gamma),
\quad
(t,\x)\in\mathcal T\times\mathcal X .
\end{equation}
Although the integral is taken over the full path space, the integrand at time
\(t\) depends only on the stopped history \(\gamma_{\cdot\wedge t}\).
Define
\begin{align}
    H^\rho(t,\x)
    :=
    \Pi_{\mathcal X}\bigl(\x-\bar\Phi^\rho(t,\x)\bigr)-\x,
    \quad
    (t,\x)\in\mathcal T\times\mathcal X .\label{eq:H_rho_t_x}
\end{align}
The first-stage trajectory in environment \(\rho\) is governed by
\begin{align}
    \dot x^\rho(t)
    =
    H^\rho(t,x^\rho(t)),
    \quad
    x^\rho(0)=\x_0^\rho .\label{eq:ODE_rho}
\end{align}
By Theorem~\ref{thm:first_stage_existence}, this ODE admits a unique solution.
The transfer-learning problem is formulated as the stability of the
solution map
$
    (\x_0,\mu)\mapsto x(\cdot;\x_0,\mu)
$
with respect to changes in the initial state and the exogenous path law.

\subsection{Metric framework for transfer stability}
\label{subsec:metric_transfer_stability}

To study the stability of the solution map
$
    (\x_0,\mu)\mapsto x(\cdot;\x_0,\mu),
$
we specify the metrics used to compare state trajectories and exogenous path
laws.

For two first-stage trajectories \(x^\alpha,x^\beta\in\mathcal C_n\), we use
the uniform distance
\[
\|x^\alpha-x^\beta\|_{\infty}
:=
\sup_{t\in\mathcal T}
\|x^\alpha(t)-x^\beta(t)\|.
\]
For the two exogenous path laws \(\mu^\alpha\) and \(\mu^\beta\) on \(\Lambda\),
we use the  Wasserstein distance
\[
W_1(\mu^\alpha,\mu^\beta)
:=
\inf_{\pi\in\Pi(\mu^\alpha,\mu^\beta)}
\int_{\Lambda\times\Lambda}
\|\gamma^\alpha-\gamma^\beta\|_{\infty}\,
\pi(d\gamma^\alpha,d\gamma^\beta),
\]
where \(\Pi(\mu^\alpha,\mu^\beta)\) denotes the set of all couplings of the two
path laws; see, e.g., \cite[Chapter~6]{villani2009optimal}. This distance will be used to quantify the similarity between the source and
target stochastic environments through their exogenous path laws.

\subsection{\texorpdfstring{$1/2$-Hölder continuity of parametric VI solutions}
{One-half Holder continuity of parametric VI solutions}}
\label{subsec:parametric_VI_holder}

We first study the parameter dependence of the second-stage VI. In the
history-dependent DSVI model, the second-stage feasible set is
$
\mathcal Y(t,\xi_t),
$
which depends on both the time $t$ and the current value of the exogenous
process. Therefore, to establish the transfer stability of the first-stage
trajectory, we need a stability estimate for the solution of a parametric VI
whose feasible set also moves with the parameter.

For parametric variational inequalities (PVIs) on moving sets, a common
approach is to impose regularity conditions that guarantee Lipschitz
continuity of the solution function. For example,
\cite{ChenZhangZhang2026MovingSet} studies optimization problems with
PVI constraints on a moving set, where both the VI mapping and the feasible set
involve the parameter, and establishes Lipschitz continuity of the PVI solution
mapping under suitable assumptions. Such Lipschitz-type conditions may be
stronger than what is needed in the transfer-stability analysis of this paper.

In the present setting, a weaker and more tractable assumption is that the
feasible-set mapping
$
(t,\xi)\mapsto \mathcal Y(t,\xi)
$
is Lipschitz continuous with respect to the Hausdorff distance. This raises the
question of whether such a Hausdorff-Lipschitz condition is sufficient to
ensure Lipschitz continuity of the parametric VI solution mapping. The following
example shows that, in general, it is not.

\begin{example}\label{ex:hausdorff_not_lipschitz}
Let
$
p=(1,0)\in\mathbb R^2
$
and consider the strongly monotone mapping
$
A(y)=y-p.
$
For $\varepsilon\ge0$, define
$
K_0=[-1,0]\times[0,2]$, $
q_\varepsilon=(\varepsilon,\sqrt{\varepsilon}),
$
and, for $\varepsilon>0$,
$
K_\varepsilon
:=
\operatorname{co}\bigl(K_0\cup\{q_\varepsilon\}\bigr).
$
Each $K_\varepsilon$ is nonempty, compact, and convex. Moreover,
$
K_0\subseteq K_\varepsilon
$
and
$
d_H(K_\varepsilon,K_0)=\varepsilon=|\varepsilon-0|.
$ Thus, the feasible-set mapping
$
\varepsilon\mapsto K_\varepsilon
$
is Lipschitz continuous at $\varepsilon=0$ in the Hausdorff distance.
Consider the VI:
$
0\in y_\varepsilon-p+N_{K_\varepsilon}(y_\varepsilon),
$
which is equivalent to the projection problem
$
y_\varepsilon=\Pi_{K_\varepsilon}(p).
$
For $\varepsilon=0$, we have
$
y_0=(0,0).
$
For $\varepsilon>0$, the projection lies on the segment joining $(0,0)$ and
$q_\varepsilon$. Hence
$
y_\varepsilon
=
\frac{\langle p,q_\varepsilon\rangle}{\|q_\varepsilon\|^2}
q_\varepsilon
=
\frac{1}{1+\varepsilon}
(\varepsilon,\sqrt{\varepsilon}).
$
Therefore
$
\|y_\varepsilon-y_0\|
=
\sqrt{\frac{\varepsilon}{1+\varepsilon}}
\le
|\varepsilon-0|^{1/2}.
$
Thus the solution mapping
$
\varepsilon\mapsto y_\varepsilon
$
is $1/2$-H\"older continuous on the parameter at $\varepsilon=0$.
\end{example}

The preceding example shows that Hausdorff--Lipschitz continuity of the
feasible sets alone does not imply Lipschitz continuity of the VI solution.
For the transfer analysis of our history-dependent model, however, Lipschitz
continuity is not necessary. The following lemma shows that, under the standing
strong monotonicity and boundedness assumptions, the second-stage response is
still $1/2$-Hölder continuous with respect to the history path.

For fixed $t\in\mathcal T$, $\xi_{\cdot\wedge t}\in\Lambda$,
$\y\in \mathcal Y(t,\xi_t)$, and $\x\in {\mathcal X}$, define
$
K_t(\xi_{\cdot\wedge t})
:=
\mathcal Y(t,\xi_t),
$
and
$
\Psi(t,\xi_{\cdot\wedge t},\x,\y)
:=
\int_0^t
R(s,\xi_s,\x,\y)\,ds
+
F(t,\xi_t,\x,\y).
$
Then the second-stage problem can be written as the following
history-parametric VI:
\begin{equation}
0
\in
\Psi(t,\xi_{\cdot\wedge t},\x,\y)
+
N_{K_t(\xi_{\cdot\wedge t})}(\y).\label{eq:history_parametric_VI}
\end{equation}

\begin{lemma}\label{lem:second_stage_holder_history}
Suppose Assumption~\ref{ass:existence_first_stage} holds. Assume further that
there exists $L_{\mathcal Y}>0$ such that, for all $t\in\mathcal T$ and
$\xi_1,\xi_2\in\Xi$,
$
d_H\bigl(\mathcal Y(t,\xi_1),\mathcal Y(t,\xi_2)\bigr)
\le
L_{\mathcal Y}\|\xi_1-\xi_2\|.
$
Then, for every fixed $t\in\mathcal T$ and $\x\in {\mathcal X}$, the solution
$\y(t,\xi_{\cdot\wedge t},\x)$ of \eqref{eq:history_parametric_VI} satisfies
\[
\begin{aligned}
&
\left\|
\y(t,\xi^\alpha_{\cdot\wedge t},\x)
-
\y(t,\xi^\beta_{\cdot\wedge t},\x)
\right\|
\le
C_Y(t)
\left(
\|\xi^\alpha-\xi^\beta\|_{\infty}
+
\|\xi^\alpha-\xi^\beta\|_{\infty}^{1/2}
\right),\
\forall \xi^\alpha,\xi^\beta\in\Lambda,
\end{aligned}
\]
where
\begin{align}
C_Y(t)
:=
\max\left\{\frac{
\int_0^t\ell_R(s)\,ds+\ell_F(t)
}{m_F}
,\
\sqrt{
\frac{2M_\Psi(t)L_{\mathcal Y}}{m_F}
}\right\},\label{C_Y_t}
\end{align}
and
$
M_\Psi(t)
:=
\sup_{\xi_{\cdot\wedge t}\in\Lambda,\ \x\in {\mathcal X},\ \y\in\overline{\mathcal Y}}
\left\|
\Psi(t,\xi_{\cdot\wedge t},\x,\y)
\right\|
<\infty .
$
\end{lemma}

\begin{proof}
Fix $t\in\mathcal T$ and $\x\in {\mathcal X}$. For simplicity, write
$
\y_\alpha:=\y(t,\xi^\alpha_{\cdot\wedge t},\x),
\
\y_\beta:=\y(t,\xi^\beta_{\cdot\wedge t},\x),
$
and
$
K_\alpha:=K_t(\xi^\alpha_{\cdot\wedge t}),
\
K_\beta:=K_t(\xi^\beta_{\cdot\wedge t}).
$
By Lemma \ref{lem:second_stage_properties}, the solution of
\eqref{eq:history_parametric_VI} is unique.
We now estimate the dependence of the solution on the history path. Set
$
d:=\|\xi^\alpha-\xi^\beta\|_{\infty}$ and $
h:=d_H(K_\alpha,K_\beta).
$
By the Hausdorff continuity of $\mathcal Y$,
\[
\begin{aligned}
h
=
d_H\bigl(
\mathcal Y(t,\xi^\alpha_t),
\mathcal Y(t,\xi^\beta_t)
\bigr)                                                 \le
L_{\mathcal Y}\|\xi^\alpha_t-\xi^\beta_t\|
\le
L_{\mathcal Y}\|\xi^\alpha-\xi^\beta\|_{\infty}
=
L_{\mathcal Y}d .
\end{aligned}
\]

By the definition of the Hausdorff distance, there exist
$\widetilde \y_\beta\in K_\alpha$ and $\widetilde \y_\alpha\in K_\beta$ such that
$
\|\widetilde \y_\beta-\y_\beta\|\le h$ and $
\|\widetilde \y_\alpha-\y_\alpha\|\le h .
$
Using $\widetilde \y_\beta\in K_\alpha$ in
\eqref{eq:history_parametric_VI}, we have
$
\left\langle
\Psi(t,\xi^\alpha_{\cdot\wedge t},\x,\y_\alpha),
\widetilde \y_\beta-\y_\alpha
\right\rangle
\ge 0.
$
Since
$
\widetilde \y_\beta-\y_\alpha
=
\y_\beta-\y_\alpha
+
\widetilde \y_\beta-\y_\beta,
$
it follows that
\begin{align}
\left\langle
\Psi(t,\xi^\alpha_{\cdot\wedge t},\x,\y_\alpha),
\y_\alpha-\y_\beta
\right\rangle
&\le
\left\langle
\Psi(t,\xi^\alpha_{\cdot\wedge t},\x,\y_\alpha),
\widetilde \y_\beta-\y_\beta
\right\rangle                               \le
M_\Psi(t)h .\label{1lem1}
\end{align}
Similarly, using $\widetilde \y_\alpha\in K_\beta$ in
\eqref{eq:history_parametric_VI}, we get
$
\left\langle
\Psi(t,\xi^\beta_{\cdot\wedge t},\x,\y_\beta),
\widetilde \y_\alpha-\y_\beta
\right\rangle
\ge 0.
$
Hence,
\begin{align}
&\quad
-\left\langle
\Psi(t,\xi^\beta_{\cdot\wedge t},\x,\y_\beta),
\y_\alpha-\y_\beta
\right\rangle                                             \nonumber\\
&=
-\left\langle
\Psi(t,\xi^\beta_{\cdot\wedge t},\x,\y_\beta),
\widetilde \y_\alpha-\y_\beta
\right\rangle
+
\left\langle
\Psi(t,\xi^\beta_{\cdot\wedge t},\x,\y_\beta),
\widetilde \y_\alpha-\y_\alpha
\right\rangle         \le
M_\Psi(t)h . \label{2lem1}
\end{align}
Let
$
e:=\y_\alpha-\y_\beta .
$
By the strong monotonicity of
$\Psi(t,\xi^\alpha_{\cdot\wedge t},\x,\cdot)$,
\[{\small
\begin{aligned}
&m_F\|e\|^2
\le
\left\langle
\Psi(t,\xi^\alpha_{\cdot\wedge t},\x,\y_\alpha)
-
\Psi(t,\xi^\alpha_{\cdot\wedge t},\x,\y_\beta),
e
\right\rangle    \\
=&
\left\langle
\Psi(t,\xi^\alpha_{\cdot\wedge t},\x,\y_\alpha),
e
\right\rangle
-
\left\langle
\Psi(t,\xi^\beta_{\cdot\wedge t},\x,\y_\beta),
e
\right\rangle                            
-
\left\langle
\Psi(t,\xi^\alpha_{\cdot\wedge t},\x,\y_\beta)
-
\Psi(t,\xi^\beta_{\cdot\wedge t},\x,\y_\beta),
e
\right\rangle .
\end{aligned}}
\]
Using \eqref{1lem1} and \eqref{2lem1}, together with
Assumption~\ref{ass:existence_first_stage}, we obtain
\begin{align}
m_F\|e\|^2
&\le
2M_\Psi(t)h
+
\left\|
\Psi(t,\xi^\alpha_{\cdot\wedge t},\x,\y_\beta)
-
\Psi(t,\xi^\beta_{\cdot\wedge t},\x,\y_\beta)
\right\|
\|e\|                         \nonumber                            \\
&\le
2M_\Psi(t)L_{\mathcal Y}d
+
\left(
\int_0^t\ell_R(s)\,ds+\ell_F(t)
\right)
d\|e\|.\label{quadratic_inequality}
\end{align}
Indeed,
\[
\begin{aligned}
&
\left\|
\Psi(t,\xi^\alpha_{\cdot\wedge t},\x,\y_\beta)
-
\Psi(t,\xi^\beta_{\cdot\wedge t},\x,\y_\beta)
\right\|                                                 \\
&\le
\int_0^t
\left\|
R(s,\xi^\alpha_s,\x,\y_\beta)
-
R(s,\xi^\beta_s,\x,\y_\beta)
\right\|\,ds
+
\left\|
F(t,\xi^\alpha_t,\x,\y_\beta)
-
F(t,\xi^\beta_t,\x,\y_\beta)
\right\|                                                 \\
&\le
\left(
\int_0^t\ell_R(s)\,ds+\ell_F(t)
\right)
\|\xi^\alpha-\xi^\beta\|_{\infty}.
\end{aligned}
\]

Solving the quadratic inequality \eqref{quadratic_inequality} yields
\[
\|e\|
\le
\frac{
\left(
\int_0^t\ell_R(s)\,ds+\ell_F(t)
\right)d
+
\sqrt{
\left(
\int_0^t\ell_R(s)\,ds+\ell_F(t)
\right)^2d^2
+
8m_FM_\Psi(t)L_{\mathcal Y}d
}
}{2m_F}.
\]
Using $\sqrt{a+b}\le\sqrt a+\sqrt b$ for $a,b\ge0$, we obtain
\[
\|e\|
\le
\frac{
\int_0^t\ell_R(s)\,ds+\ell_F(t)
}{m_F}d
+
\sqrt{
\frac{2M_\Psi(t)L_{\mathcal Y}}{m_F}
}
d^{1/2}.
\]
Thus, we have
$
\left\|
\y(t,\xi^\alpha_{\cdot\wedge t},\x)
-
\y(t,\xi^\beta_{\cdot\wedge t},\x)
\right\|
\le
C_Y(t)
\left(
\|\xi^\alpha-\xi^\beta\|_{\infty}
+
\|\xi^\alpha-\xi^\beta\|_{\infty}^{1/2}
\right).
$
The proof is complete.
\end{proof}

\subsection{Transfer stability of the first-stage trajectory}
\label{subsec:transfer_stability_x}

We now return to the history-dependent DSVI model
\eqref{eq:closed_loop_model_projected_random}. As shown in the preceding
subsection, Hausdorff-Lipschitz continuity of the feasible-set mapping
generally leads to a $1/2$-Hölder dependence of the
parametric VI solution. Motivated by this observation, we formulate the
following assumption for the second-stage response and then establish the
stability of the first-stage trajectory with respect to the exogenous path
distribution and the initial state.

\begin{theorem}\label{thm:transfer_stability_x}
Suppose that the assumptions of Theorem~\ref{thm:first_stage_existence} and
Lemma~\ref{lem:second_stage_holder_history} hold. For
\(\rho\in\{\alpha,\beta\}\), let \(x^\rho\) be the unique solution of
\eqref{eq:ODE_rho}.
Then there exists a constant $C_T>0$ such that
\[
\|x^\alpha-x^\beta\|_{\infty}
\le
C_T
\left(
\|\x_0^\alpha-\x_0^\beta\|
+
W_1(\mu^\alpha,\mu^\beta)
+
W_1(\mu^\alpha,\mu^\beta)^{1/2}
\right).
\]
\end{theorem}

\begin{proof}
Fix $t\in\mathcal T$ and $\x\in {\mathcal X}$. Let
$\xi^\alpha,\xi^\beta\in\Lambda$ be two admissible paths, and denote by
$
\y(t,\xi^\alpha_{\cdot\wedge t},\x)
$
and
$
\y(t,\xi^\beta_{\cdot\wedge t},\x)
$
the unique second-stage VI solutions of \eqref{eq:history_parametric_VI} corresponding to the paths
$\xi^\alpha$ and $\xi^\beta$, respectively.
By Lemma~\ref{lem:second_stage_holder_history}, we have
\[
\left\|
\y(t,\xi^\alpha_{\cdot\wedge t},\x)
-
\y(t,\xi^\beta_{\cdot\wedge t},\x)
\right\| \le
C_Y(t)
\left(
\|\xi^\alpha-\xi^\beta\|_{\infty}
+
\|\xi^\alpha-\xi^\beta\|_{\infty}^{1/2}
\right),
\]
where $C_Y(t)$ is the nonnegative function defined in \eqref{C_Y_t}.
Meanwhile,  Lemma~\ref{lem:second_stage_properties} gives the
Lipschitz dependence of the second-stage response on \(\x\) with modulus
\(\kappa_Y(t):=m_F^{-1}\bigl(\ell_F(t)+\int_0^t\ell_R(s)\,ds\bigr)\).

We next estimate the expected-value mapping in \eqref{bar_phi_rho}. For each
$\rho\in\{\alpha,\beta\}$, Assumption~\ref{ass:existence_first_stage}
\textup{(A3)} and the estimate \eqref{bar_phi_estimate} give
\[
\|\bar\Phi^\rho(t,\x_1)-\bar\Phi^\rho(t,\x_2)\|
\le
\ell_\Phi(t)(1+\kappa_Y(t))
\|\x_1-\x_2\|.
\]

Let $\pi\in\Pi(\mu^\alpha,\mu^\beta)$. For a fixed $\x\in {\mathcal X}$, by
Assumption~\ref{ass:existence_first_stage}\textup{(A3)},
\[
\begin{aligned}
&\quad
\|\bar\Phi^\alpha(t,\x)-\bar\Phi^\beta(t,\x)\|                 \\
&\le
\int_{\Lambda\times\Lambda}
\Bigl\|
\Phi\bigl(t,\xi^\alpha_t,\x,
\y(t,\xi^\alpha_{\cdot\wedge t},\x)\bigr)
-
\Phi\bigl(t,\xi^\beta_t,\x,
\y(t,\xi^\beta_{\cdot\wedge t},\x)\bigr)
\Bigr\|
\,\pi(d\xi^\alpha,d\xi^\beta)                           \\
&\le
\ell_\Phi(t)
\int_{\Lambda\times\Lambda}
\left[
\|\xi^\alpha-\xi^\beta\|_{\infty}
+
C_Y(t)
\left(
\|\xi^\alpha-\xi^\beta\|_{\infty}
+
\|\xi^\alpha-\xi^\beta\|_{\infty}^{1/2}
\right)
\right]
\,\pi(d\xi^\alpha,d\xi^\beta).
\end{aligned}
\]
Applying Jensen's inequality 
to the
concave function $r\mapsto r^{1/2}$ on $\mathbb R_+$, we have
\[
\int_{\Lambda\times\Lambda}
\|\xi^\alpha-\xi^\beta\|_{\infty}^{1/2}
\,d\pi
\le
\left(
\int_{\Lambda\times\Lambda}
\|\xi^\alpha-\xi^\beta\|_{\infty}
\,d\pi
\right)^{1/2}.
\]
Taking the infimum over all couplings $\pi\in\Pi(\mu^\alpha,\mu^\beta)$ gives
\[
\sup_{\x\in {\mathcal X}}
\|\bar\Phi^\alpha(t,\x)-\bar\Phi^\beta(t,\x)\|
\le
\ell_\Phi(t)
\left[
(1+C_Y(t))W_1(\mu^\alpha,\mu^\beta)
+
C_Y(t)W_1(\mu^\alpha,\mu^\beta)^{1/2}
\right].
\]
By \eqref{eq:H_rho_t_x} and the nonexpansiveness of
\(\Pi_{\mathcal X}\), we have
\(\|H^\rho(t,\x_1)-H^\rho(t,\x_2)\|
\le B(t)\|\x_1-\x_2\|\), where
\(B(t):=2+\ell_\Phi(t)\bigl(1+\kappa_Y(t)\bigr)\).
Moreover,
\[
\sup_{\x\in {\mathcal X}}
\|H^\alpha(t,\x)-H^\beta(t,\x)\|
\le
D(t)
\left(
W_1(\mu^\alpha,\mu^\beta)
+
W_1(\mu^\alpha,\mu^\beta)^{1/2}
\right),
\]
where we use $C_Y(t)\le 1+C_Y(t)$, and
\begin{align}
D(t)
:=
\ell_\Phi(t)(1+C_Y(t)).\label{eq:D_t}
\end{align}

Since
$
x^\rho(t)
=
\x_0^\rho
+
\int_0^t H^\rho(s,x^\rho(s))\,ds,
\
\rho\in\{\alpha,\beta\},
$
we have
\[
\begin{aligned}
\|x^\alpha(t)-x^\beta(t)\|
&\le
\|\x_0^\alpha-\x_0^\beta\|+
\int_0^t
\|H^\alpha(s,x^\alpha(s))-H^\beta(s,x^\beta(s))\|\,ds .
\end{aligned}
\]
For the integrand, adding and subtracting
$H^\alpha(s,x^\beta(s))$ gives
\[
\begin{aligned}
&\quad
\|H^\alpha(s,x^\alpha(s))-H^\beta(s,x^\beta(s))\|        \\
&\le
B(s)\|x^\alpha(s)-x^\beta(s)\|+
D(s)
\left(
W_1(\mu^\alpha,\mu^\beta)
+
W_1(\mu^\alpha,\mu^\beta)^{1/2}
\right).
\end{aligned}
\]
Therefore, for every $t\in\mathcal T$,
\[
\begin{aligned}
\|x^\alpha(t)-x^\beta(t)\|
&\le
\|\x_0^\alpha-\x_0^\beta\|+
\int_0^t
B(s)\|x^\alpha(s)-x^\beta(s)\|\,ds                 \\
&\quad+
\left(
W_1(\mu^\alpha,\mu^\beta)
+
W_1(\mu^\alpha,\mu^\beta)^{1/2}
\right)
\int_0^tD(s)\,ds .
\end{aligned}
\]
By Gronwall's inequality \cite[Lemma~2.7]{teschl2012ordinary}, we obtain
\[
\begin{aligned}
&\quad\|x^\alpha(t)-x^\beta(t)\|\\
&\le
\left[
\|\x_0^\alpha-\x_0^\beta\|
+
\left(
W_1(\mu^\alpha,\mu^\beta)
+
W_1(\mu^\alpha,\mu^\beta)^{1/2}
\right)
\int_0^T D(s)\,ds
\right]
\exp\left(\int_0^T B(s)\,ds\right).
\end{aligned}
\]
Taking the supremum over $t\in\mathcal T$ yields
\[
\|x^\alpha-x^\beta\|_{\infty}
\le
C_T
\left(
\|\x_0^\alpha-\x_0^\beta\|
+
W_1(\mu^\alpha,\mu^\beta)
+
W_1(\mu^\alpha,\mu^\beta)^{1/2}
\right),
\]
where one can take
$
C_T
:=
\exp\left(\int_0^T B(s)\,ds\right)
\max\left\{
1,\int_0^T D(s)\,ds
\right\}.
$
\end{proof}

\begin{remark}\label{rem:transfer_interpretation}
Theorem~\ref{thm:transfer_stability_x} gives Lipschitz stability with respect
to the initial state and local \(1/2\)-Hölder stability with respect to the law
of the exogenous process.
Thus, similar source and target environments and initial states yield close
first-stage trajectories, providing a quantitative justification for transfer
learning in the history-dependent DSVI model.
\end{remark}

\section{Numerical Analysis}
\label{sec:numerical_analysis}

This section presents a numerical example to illustrate the convergence of the
SAA approximation and the stability of transfer learning. The numerical test is
conducted in MATLAB 2025a on a Lenovo desktop with a 2.60GHz CPU and 32.0GB RAM.
We set
$
\mathcal T=[0,12]
$,
$
n=m=3
$,
and
$
{\mathcal X}=[-2.4,2.4]^3.
$
The time step is
$
\Delta t=0.05.
$
The exogenous process is
$
\bm\xi_t=(\bm\xi_1(t),\bm\xi_2(t))\in\mathbb R^2.
$
For each sample path $j$, the two coordinates are generated by
$
\bm\xi_1^j(t)
=
2.45a_1^j\sin(0.70t+p_1^j),
\ 
\bm\xi_2^j(t)
=
2.35a_2^j\cos(0.65t+p_2^j),
$
where
$
a_i^j\sim N(0,1)
$
and
$
p_i^j\sim U(0,2\pi)
$
are independent.
For the second-stage VI \eqref{eq:closed_loop_model_projected_random_second}, we choose
\[
F(t,\xi,\x,\y)
=
A_F(t,\xi,\x)\y
+
\beta_F d_\y^\top\tanh(d_\y \y)
+
b_F(t,\xi,\x),
\]
\[
A_F(t,\xi,\x)
=
\begin{pmatrix}
2.70 & 0.10\cos(0.30t) & -0.08\tanh(\xi_1)\\
0.10\cos(0.30t) & 2.68 & 0.08\sin(\xi_2)\\
-0.08\tanh(\xi_1) & 0.08\sin(\xi_2) & 2.72
\end{pmatrix},
\]
\[
b_F(t,\xi,\x)
=
\begin{pmatrix}
0.25\tanh(x_1+\xi_1)\\
0.25\tanh(x_2+\xi_2)\\
0.22\tanh(x_3+\xi_1-\xi_2)
\end{pmatrix},
\
\beta_F=0.30,
\
d_\y=(0.80,-0.45,0.35).
\]
The history-dependent mapping is
\[
R(s,\xi,\x,\y)
=
A_R(s,\xi,\x)\y
+
\beta_R d_\y^\top\tanh(d_\y \y)
+
b_R(s,\xi,\x),
\]
\[
A_R(s,\xi,\x)
=
\begin{pmatrix}
0.60 & 0.03\sin(0.25s) & -0.03\tanh(\xi_1)\\
0.03\sin(0.25s) & 0.62 & 0.03\cos(\xi_2)\\
-0.03\tanh(\xi_1) & 0.03\cos(\xi_2) & 0.61
\end{pmatrix},
\]
\[
b_R(s,\xi,\x)
=
\begin{pmatrix}
0.14\tanh(x_1)+0.14\tanh(\xi_1)\\
0.14\tanh(x_2)+0.14\tanh(\xi_2)\\
0.12\tanh(x_3)+0.12\tanh(\xi_1-\xi_2)
\end{pmatrix},
\
\beta_R=0.12.
\]
The feasible set is the moving box
$
\mathcal Y(t,\xi)
=
[c(t,\xi)-r,\ c(t,\xi)+r],
$
where
\[
c(t,\xi)
=
\begin{pmatrix}
0.65\tanh(\xi_1)\\
0.65\tanh(\xi_2)\\
0.60\tanh(\xi_1-\xi_2)
\end{pmatrix},
\quad
r=(1.30,1.25,1.30)^\top .
\]
Hence
$
\mathcal Y(t,\xi)\subset \overline{\mathcal Y}
$
with
$
\overline{\mathcal Y}=[-2,2]^3.
$
The first-stage mapping in \eqref{eq:closed_loop_model_projected_random_first} is chosen as
\[
\Phi(t,\xi,\x,\y)
=
\begin{pmatrix}
0.45\sin(2\pi t/12)
+
0.90\sin(0.45x_1+0.75y_1+0.85\xi_1)
\\
0.42\cos(2\pi t/12)
+
0.88\sin(0.45x_2+0.75y_2+0.85\xi_2)
\\
0.44\sin(2\pi t/12+\pi/3)
+
0.86\sin(0.45x_3+0.75y_3+0.70(\xi_1-\xi_2))
\end{pmatrix}.
\]

It is clear that   Assumption~\ref{ass:existence_first_stage} holds. 
Since $c(t,\xi)$ is continuous and the sample paths of $\bm\xi$ are continuous, the graph of $(t,\omega)\mapsto\mathcal Y(t,\bm\xi_t(\omega))$ is measurable and the moving box $\mathcal Y(t,\xi)$ is Hausdorff-Lipschitz continuous with
respect to $\xi$; hence the condition in Lemma~\ref{lem:second_stage_holder_history} is satisfied. Consequently, by Theorem~\ref{thm:first_stage_existence}, the closed-loop
system admits a unique solution
$
x\in AC(\mathcal T;{\mathcal X})
$
for every
$
\x_0\in {\mathcal X}
$.

In the numerical experiments, the true solution of \eqref{Expected_system} is represented by a
large-sample reference trajectory $x_{\rm ref}(t)$, computed with
$N_{\rm ref}=3500$ sampled realizations
$\xi^1,\dots,\xi^{N_{\rm ref}}$ of the exogenous process $\bm\xi$. For each sampled
history $\xi^j_{\cdot\wedge t}$ and state $\x$, the response
$\y(t,\xi^j_{\cdot\wedge t},\x)$ is obtained by solving the second-stage VI
\eqref{eq:closed_loop_model_projected_random_second}. Specifically,
\[
\bar\Phi_{\rm ref}(t,\x):=N_{\rm ref}^{-1}
\sum_{j=1}^{N_{\rm ref}}
\Phi(t,\xi_t^j,\x,\y(t,\xi^j_{\cdot\wedge t},\x)),
\]
and, for each sample size
$N$, $\bar\Phi_N(t,\x):=N^{-1}
\sum_{j=1}^{N}
\Phi(t,\xi_t^j,\x,\y(t,\xi^j_{\cdot\wedge t},\x))$. The SAA trajectory
$x_N(t)$ is computed from $\bar\Phi_N$. All Monte Carlo experiments use $50$ repetitions.

Figure~\ref{fig:SAA_drift} reports the SAA convergence of the expected-value
mapping. For computational tractability, the uniform error over ${\mathcal X}$ of the SAA error is approximated on
the finite grid
$
{\mathcal X}_{\rm grid}
=
\{-2.4,-1.2,0,1.2,2.4\}^3,
$
and define
\[
E_N:
=
\sup_{t_k,\x\in {\mathcal X}_{\rm grid}}
\left\|
\bar\Phi_N(t_k,\x)-\bar\Phi_{\rm ref}(t_k,\x)
\right\|.
\]
The left panels compare the reference mapping and its SAA
approximations at
$
\x=(0,0,0)^\top
$.
The upper-right panel gives the empirical tail probability
$
\mathrm P\{E_N>\varepsilon\}
$,
and the lower-right panel gives the mean, median, and interquartile envelope of
$
E_N
$.
The decay of these curves illustrates the uniform convergence of the SAA.

\begin{figure}[htbp]
\centering
\includegraphics[width=0.85\textwidth]{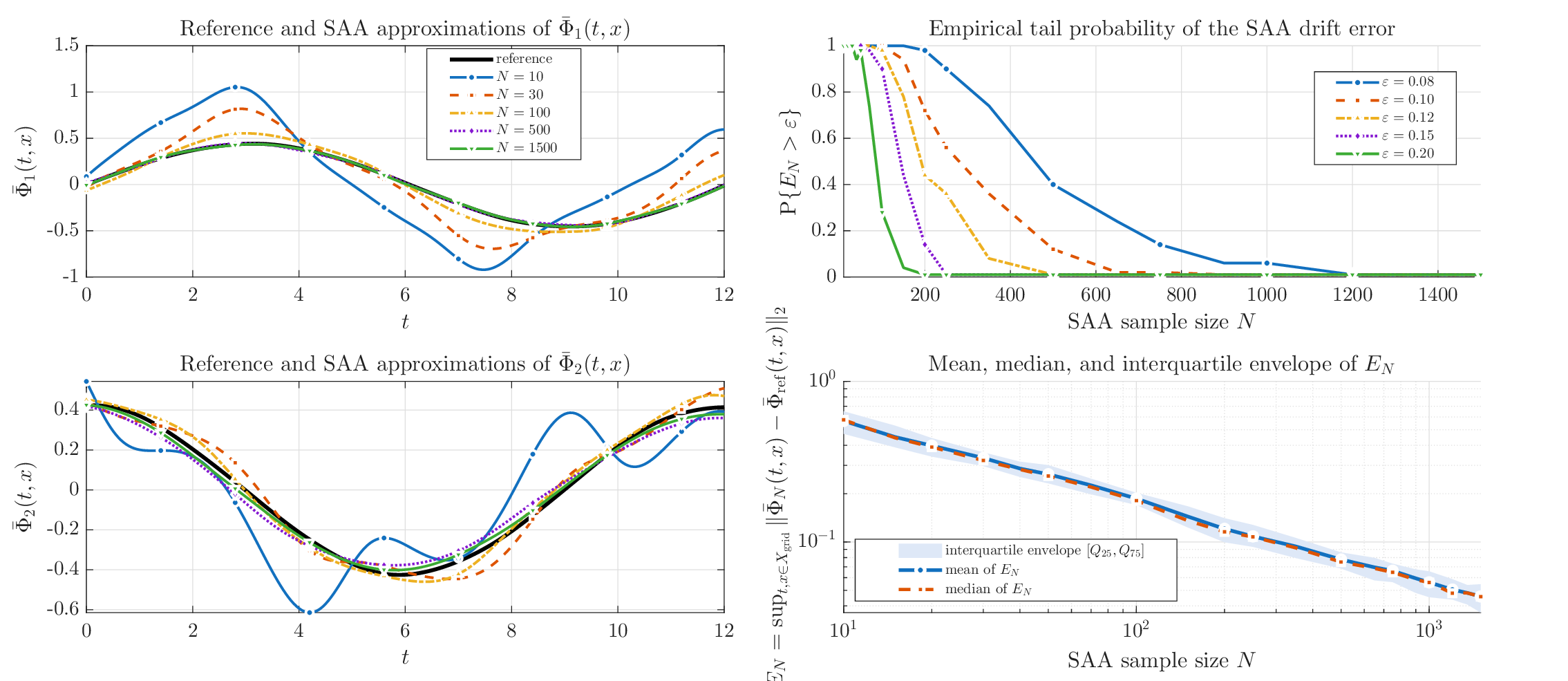}
\caption{SAA uniform convergence of the expected-value mapping.}
\label{fig:SAA_drift}
\end{figure}

Figure~\ref{fig:SAA_state} shows the convergence of the closed-loop state
trajectory. The first three panels display the three components of
$
x_N(t)
$
for different sample sizes, together with the reference trajectory. The last
panel plots
$
\sup_{t_k>0}
\|x_N(t_k)-x_{\rm ref}(t_k)\|.
$
As $N$ increases, the SAA trajectory approaches the reference solution, 
consistent with Theorem~\ref{thm:SAA_convergence}.

\begin{figure}[htbp]
\centering
\includegraphics[width=0.85\textwidth]{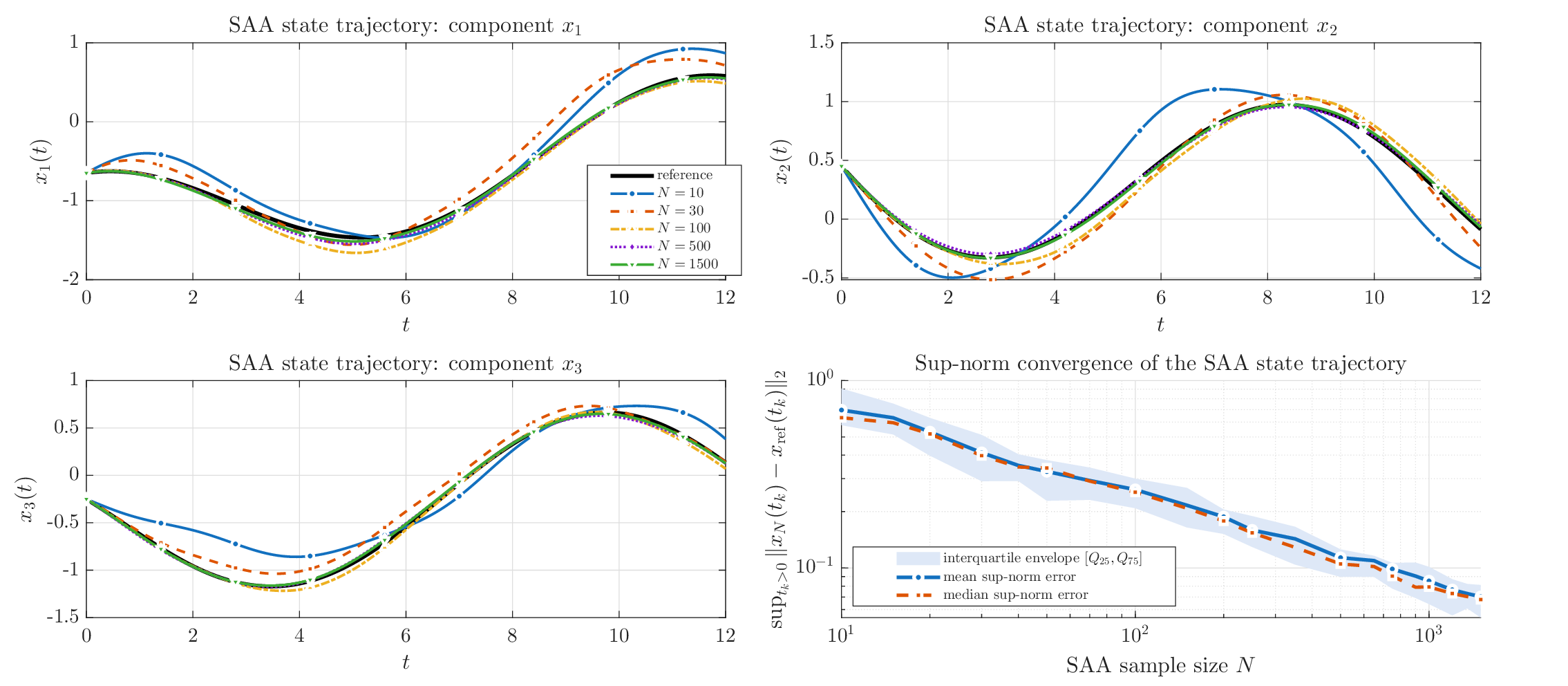}
\caption{SAA state trajectories and uniform convergence of the three-dimensional state.}
\label{fig:SAA_state}
\end{figure}

For transfer stability, we use
$
N=2000
$
sample paths. The initial state is perturbed by
$
\x_0^\beta
=
\x_0^\alpha+\delta v,
$
where $v$ is a random unit vector. The exogenous process is perturbed by a
single deterministic path direction:
$
\bm\xi^{\beta,j}(t)
=
\bm\xi^{\alpha,j}(t)
+
\Delta
\begin{pmatrix}
\sin(0.40t)\\
\cos(0.40t)
\end{pmatrix}.
$
The empirical paired coupling cost
$
\widehat W_{1,N}
=
\frac1N
\sum_{j=1}^N
\sup_{t\in[0,T]}
\left\|
\bm\xi^{\beta,j}(t)-\bm\xi^{\alpha,j}(t)
\right\|
$
is approximately equal to the perturbation level
$
\Delta.
$
In the experiment, we choose
$\Delta\in
\texttt{logspace}(-4,\log_{10}(0.20),18).
$

Figure~\ref{fig:transfer_stability} illustrates the transfer stability. The left
panel shows the dependence of
\[
\sup_{t_k>1}
\|x^\alpha(t_k)-x^\beta(t_k)\|
\]
on the initial-state perturbation
$
\|\x_0^\alpha-\x_0^\beta\|.
$
The right panel uses
$
\widehat W_{1,N}^{1/2}
$
as the horizontal axis and shows the corresponding stability with respect to
the exogenous distribution. Both panels show decreasing errors as the
perturbation level decreases, in agreement with the transfer stability
estimate.

\begin{figure}[htbp]
\centering
\includegraphics[width=0.9\textwidth]{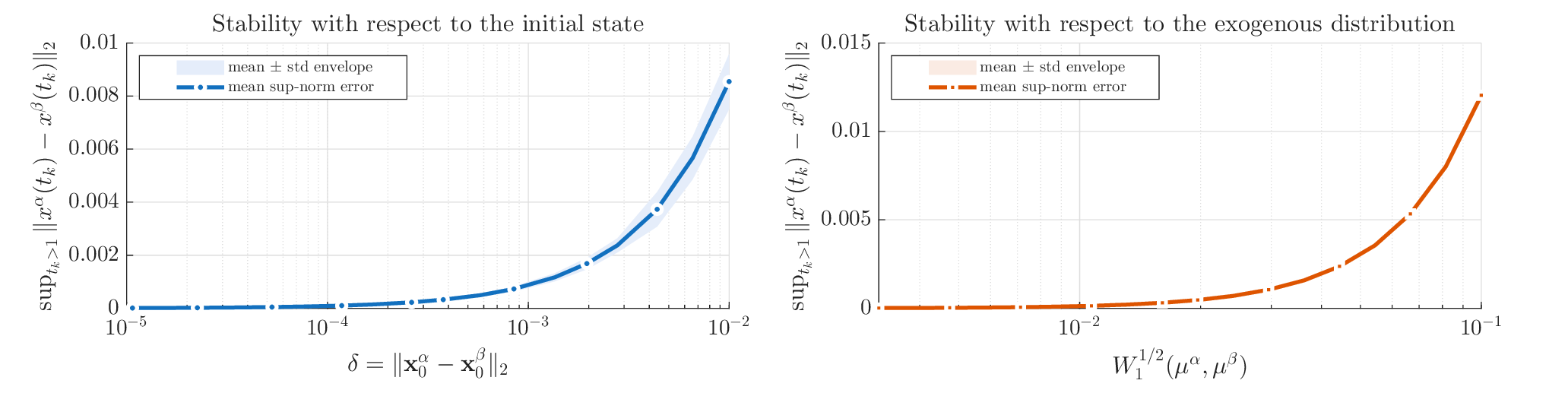}
\caption{Transfer stability under initial-state and exogenous-distribution perturbations.}
\label{fig:transfer_stability}
\end{figure}

\section{Computational Study on History-Dependent Transfer Learning}
\label{sec:health_monitoring_application}

In this section, we apply the history-dependent DSVI system
\eqref{eq:closed_loop_model_projected_random}--%
\eqref{eq:closed_loop_model_projected_random_opt}
to a synthetic source--target benchmark for elderly-health monitoring adapted
from \cite{chen2025differential}.  In this context, 
the first-stage equation
\eqref{eq:closed_loop_model_projected_random_first} describes the evolution of each individual's continuous health state $(x(t))_{t\in\mathcal T}$.
The process \(\bm{\xi}\) represents external
factors affecting the individual's health evolution.
For each history
\(\bm{\xi}_{\cdot\wedge t}\), the second-stage problem estimates the response
variable \(y(t,\bm{\xi}_{\cdot\wedge t})\), based on  the accumulated fitting error  \(R\) and the terminal fitting error \(F\). 
The two-stage structure allows
learning the history-dependent feedback component of \(\Phi\), which drives the
projected dynamics of the continuous health state. 

The goal of the monitoring system is to provide instantaneous medical response to new elderly users’ health states
based on the existing system which has been well trained with a larger and richer historical dataset from elderly users
that share similar features with the new ones.
 The analysis in the previous sections regarding the uniqueness of the solution to the second-stage SVI and the stability of the induced first-stage trajectory 
  allows us to  quickly identify  the health state of an individual based on  health state of ``similar'' elderly users in the existing system.
 In our experiments, we focus on examining the stability of the
health-state trajectory under sensor perturbations and the transferability of
the response trajectories to new users.

Let $\mathcal U_S=\{u_1^S,\ldots,u_{10}^S\}$ and
$\mathcal U_T=\{u_1^T,\ldots,u_{10}^T\}$ denote the source and target cohorts, each consisting of ten patients.
Each patient has $100$ days of smartwatch, intelligent-insole, and EMR
observations, together with health-state labels sampled every five seconds over 24 hours.
The target users are generated independently, representing a different retirement community with demographic, health, and living conditions similar to those of the source users.
 The data description
and computational materials are available at
\url{https://github.com/DSVI2025/DSVI-HD}.

For $D\in\{S,T\}$, $u\in\mathcal U_D$, modality $i\in\{1,2,3\}$, and
$\nu=1,\ldots,17280$, let
$Z_{i,D}^{u}(t_\nu)\in\mathbb R^{100\times m_i}$ and
$\eta_D^{u}(t_\nu)\in\mathbb R^{100}$ denote the feature matrix and label
vector, where $(m_1,m_2,m_3)=(14,17,30)$. The rows correspond to observation
days. For each source patient, days $1$ to $90$ are used as reference data and days
$91$ to $100$ for testing, so that $d_j^S=90+j$, $j=1,\ldots,10$. For each
target patient, ten days $d_1^T,\ldots,d_{10}^T$ are randomly selected for
evaluation and the remaining days are used as reference data. On evaluation
day $d_j^D$, write
$B_{i,D}^{u,j}(t_\nu)=\bigl(Z_{i,D}^{u}(t_\nu)\bigr)_{[d_j^D,:]}$ and
$L_D^{u,j}(t_\nu)=\bigl(\eta_D^{u}(t_\nu)\bigr)_{[d_j^D]}$.

At each time $t_\nu$, we independently sample $50$ reference days $20$ times,
using the same indices for the features and labels. For sample
$\iota=1,\ldots,20$, let
$H_{i,D}^{u,\iota}(t_\nu)\in\mathbb R^{50\times m_i}$ and
$c_D^{u,\iota}(t_\nu)\in\mathbb R^{50}$ denote the sampled features and labels,
and $\bar{\xi}_{D,u}^{\iota,t_\nu}$ denote the corresponding observation
history. To adjust the historical labels according to their similarity to the
current observation, define

\[
\resizebox{0.8\linewidth}{!}{$\displaystyle
N_{i,D}^{u,j,\iota}(t_r,t_\nu)
=
\lambda_i
\left[
\mathbf 1_{50}
-
\operatorname{softmax}
\left(
-\frac{
\operatorname{pdist2}
\bigl(
B_{i,D}^{u,j}(t_\nu),
H_{i,D}^{u,\iota}(t_r)
\bigr)
}{
\bigl(h_{i,D}^{u,\iota}(t_r)\bigr)^2
}
\right)
\right],
$}
\]

\noindent where $h_{i,D}^{u,\iota}(t_r)$ is the median pairwise distance between the rows
of $H_{i,D}^{u,\iota}(t_r)$, and $\lambda_i>0$ controls the correction size.
Thus, historical observations closer to the current observation receive
smaller corrections.

We set $\mathcal X=[0,2]$ and
$\mathcal Y=\mathcal Y_1\times\mathcal Y_2\times\mathcal Y_3$, with
$\mathcal Y_i=[-10,10]^{m_i}$. For an evaluation day $d_j^D$, let
$x_{\nu,D}^{u,j}\in\mathcal X$ denote the health-state score. The response for
modality $i$ in \eqref{eq:closed_loop_model_projected_random_opt} is obtained from
\[
\resizebox{0.8\linewidth}{!}{$\displaystyle
\begin{aligned}
\mathbf y_{i,D}^{u,j}
\bigl(t_\nu,\bar{\xi}_{D,u}^{\iota,t_\nu}\bigr)
\in
\argmin_{\mathbf y_i\in\mathcal Y_i}
\Bigg\{
&\frac{1}{2}
\sum_{r=0}^{\nu}
\left\|
H_{i,D}^{u,\iota}(t_r)\mathbf y_i
+
N_{i,D}^{u,j,\iota}(t_r,t_\nu)x_{\nu,D}^{u,j}
-
c_D^{u,\iota}(t_r)
\right\|^2
\\
&+
\frac{\rho_i}{2}
\left\|
B_{i,D}^{u,j}(t_\nu)\mathbf y_i
-
x_{\nu,D}^{u,j}
\right\|^2
\Bigg\}.
\end{aligned}
$}
\]

\noindent The first term is a cumulative linear-regression loss based on the corrected historical labels and incorporates the observation history up to $t_\nu$. The second term is a current-time fitting loss that aligns the response estimated from the current feature vector with the health-state score. The parameter $\rho_i>0$  balances the historical fitting and current-state consistency. The source-domain response trajectories are computed once for the ten source test days and retained for the transfer experiments.

The first-stage mapping in~\eqref{eq:closed_loop_model_projected_random_first}
takes the form

\[
\resizebox{0.8\linewidth}{!}{$\displaystyle
\begin{aligned}
\Phi_D^{u,j}
\bigl(
t_\nu,\bar{\xi}_{D,u}^{\iota,t_\nu},x,\mathbf y
\bigr)
={}&
A(t_\nu)x+p(t_\nu)-q(t_\nu)
\\
&-
\left[
\ell_0
+
\Delta\ell_1(t_\nu)
\left(
\sum_{i=1}^{3}
\alpha_i B_{i,D}^{u,j}(t_\nu)\mathbf y_i
+
\varepsilon_2
\bigl(t_\nu,\bar{\xi}_{D,u}^{\iota,t_\nu}\bigr)
\right)
\right],
\end{aligned}
$}
\]

\noindent where $\mathbf y=(\mathbf y_1,\mathbf y_2,\mathbf y_3)$,  $A(t_\nu)x$ and $p(t_\nu)$ describe the intrinsic damage--repair dynamics of the health state \cite{rockwood2007frailty},  $q(t_\nu)$ represents the circadian component and captures periodic fluctuations associated with the daily activity cycle~\cite{kim2023efficient}, and the last term represents the latent exogenous drive inferred from physical activity, gait stability, clinical information, and other external factors~\cite{nazaretEtAl23}.

Collecting the modality-specific responses in
$\mathbf y_D^{u,j}(t_\nu,\bar{\xi}_{D,u}^{\iota,t_\nu})$, the health-state
score is updated by
\[
\begin{aligned}
x_{\nu+1,D}^{u,j}
=
\Pi_{[0,2]}
\Bigg(
x_{\nu,D}^{u,j}
-
\frac{1}{20}
\sum_{\iota=1}^{20}
\Phi_D^{u,j}
\Bigl(
t_\nu,
\bar{\xi}_{D,u}^{\iota,t_\nu},
x_{\nu,D}^{u,j},
\mathbf y_D^{u,j}
\bigl(t_\nu,\bar{\xi}_{D,u}^{\iota,t_\nu}\bigr)
\Bigr)
\Bigg).
\end{aligned}
\]
The average approximates the history-dependent feedback, and the projection
keeps the score in $[0,2]$.

\subsection{Model setting, experimental configuration}
\label{subsec:application_model_setting}

We now specify the numerical parameters used in the computational study. We
set $A(t)\equiv\exp(-3)$, $p(t)\equiv0.01$,
$
q(t)
=
0.15\sin (2\pi\cdot0.025t)
+
0.17\sin (2\pi\cdot0.021t),
$
$\varepsilon_2(t,\xi)\equiv0$, $\ell_0=0.02$,
$\Delta\ell_1(t)\equiv1.5$,
$(\alpha_1(t),\alpha_2(t),\alpha_3(t))\equiv(0.4,0.4,0.2)$,
$\lambda_i=0.50$, and $\rho_i=5$ for $i=1,2,3$. The initial state is set
equal to the true label at the initial time,
$x_{0,D}^{u,j}=L_D^{u,j}(t_0)$.

The predicted label is obtained by discretizing the health-state score:
\[
\widehat L_D^{u,j}(t_\nu)
=
\begin{cases}
0, & x_{\nu,D}^{u,j}\le 2/3,\\
1, & 2/3<x_{\nu,D}^{u,j}<4/3,\\
2, & x_{\nu,D}^{u,j}\ge 4/3.
\end{cases}
\]

\subsubsection{Baseline results} We first evaluate the method on the held-out source-domain data.
Table~\ref{tab:source_performance} presents the class-wise accuracy, precision,
recall, specificity, and F1-score. The classification metrics are
computed following Section~5.4.1 of \cite{chen2025differential}. The history-dependent DSVI method
performs consistently well, achieving F1-scores of 0.9848, 0.9744, and 0.9974
for the healthy, weak, and ill classes, respectively. The high recall of
0.9913 for the weak class indicates that most intermediate health states are
correctly identified, while the ill class achieves near-perfect classification
with an accuracy of 0.9993 and a recall of 0.9996. These results demonstrate
the effectiveness of the proposed DSVI model in distinguishing different health
states from the source-domain multimodal data.

\begin{table}[htbp]
\caption{Class-wise held-out performance on the source-domain data.}
\label{tab:source_performance}
\centering
\setlength{\tabcolsep}{5pt}
\begin{tabular}{lccccc}
\toprule
Class & Acc. & Prec. & Recall & Spec. & F1-score \\
\midrule
healthy (0) & 0.9835 & 0.9955 & 0.9743 & 0.9947 & 0.9848 \\
weak (1)    & 0.9833 & 0.9582 & 0.9913 & 0.9795 & 0.9744 \\
ill (2)     & 0.9993 & 0.9953 & 0.9996 & 0.9993 & 0.9974 \\
\bottomrule
\end{tabular}
\end{table}

We next evaluate the target-domain performance under different response-update
frequencies. In the full-update setting, the second-stage response is
recomputed from the target reference data at every five-second time point,
providing a target-domain benchmark with dense reference information. To examine whether the computed responses can be reused over time, we also
update them every $10$ minutes, $20$ minutes, and $1$ hour. Between consecutive
updates, the most recently computed response is retained, while the state
trajectory and the evaluation metrics remain on the original five-second grid.

Table~\ref{tab:target_recomputation_update_baselines} shows that performance is consistently high across all three classes.  Furthermore, it indicates that response trajectories can be reused over appropriate time
intervals with limited loss of accuracy.  This temporal reuse provides the
basis for the transfer learning in which precomputed
source-domain responses can be reused when target-domain data are
limited. Such an effect of the response-refresh frequency will be further evaluated in
Subsection~\ref{subsec:application_temporal_transfer}.

\begin{table}[htbp]
\caption{Target-domain performance under different response-update
frequencies, with the full-update setting achieving an overall accuracy of
$97.928\%$.}
\label{tab:target_recomputation_update_baselines}
\centering
\setlength{\tabcolsep}{3pt}
\begin{tabular}{llccccc}
\toprule
Response update & Class & Acc. & Prec. & Recall & Spec. & F1-score \\
\midrule
Every 5 sec
& healthy (0) & 0.9800 & 0.9954 & 0.9674 & 0.9947 & 0.9812 \\
& weak (1)   & 0.9796 & 0.9484 & 0.9902 & 0.9745 & 0.9689 \\
& ill (2)    & 0.9990 & 0.9931 & 0.9996 & 0.9989 & 0.9964 \\
\midrule
Every 10 min
& healthy (0) & 0.9632 & 0.9696 & 0.9618 & 0.9649 & 0.9657 \\
& weak (1)   & 0.9645 & 0.9474 & 0.9420 & 0.9752 & 0.9447 \\
& ill (2)    & 0.9928 & 0.9545 & 0.9965 & 0.9922 & 0.9750 \\
\midrule
Every 20 min
& healthy (0) & 0.9529 & 0.9608 & 0.9513 & 0.9548 & 0.9560 \\
& weak (1)   & 0.9553 & 0.9423 & 0.9173 & 0.9733 & 0.9296 \\
& ill (2)    & 0.9842 & 0.9039 & 0.9928 & 0.9827 & 0.9463 \\
\midrule
Every 1 h
& healthy (0) & 0.9215 & 0.9309 & 0.9225 & 0.9203 & 0.9267 \\
& weak (1)   & 0.9302 & 0.9377 & 0.8386 & 0.9736 & 0.8854 \\
& ill (2)    & 0.9546 & 0.7652 & 0.9766 & 0.9510 & 0.8581 \\
\bottomrule
\end{tabular}
\end{table}

\subsubsection{Robustness with respect to noise}
\label{subsec:application_noise_stability}

We then examine the robustness of the source-domain health state predictions to
measurement perturbations. For a representative source user, noise is applied
entrywise to the smartwatch and intelligent-insole features
$Z_{1,S}^{u}(t_\nu)$ and $Z_{2,S}^{u}(t_\nu)$ in both the reference and
held-out data. The EMR features $Z_{3,S}^{u}(t_\nu)$ are left unchanged because
they represent static clinical records rather than continuously measured
signals.

We consider additive, multiplicative, Laplace, drift, impulse, cumulative, and
feature-dependent mixed noise. The parameter $s>0$ controls the perturbation
level. For continuous features, the noise magnitude is scaled by the empirical
standard deviation of the corresponding feature over time. The scale
coefficients used in the experiments are
$
\sigma_{\rm add}=0.025s,
\sigma_{\rm drift}=0.015s,
\sigma_{\rm mult}=0.030s,
\sigma_{\rm cum}=0.040s.
$
The Laplace perturbation is chosen to have a root mean square magnitude
comparable to that of the additive Gaussian noise. For impulse noise, the
perturbation is activated by a Bernoulli mask with probability
$p_s=\min\{0.05,0.0025s\}$. Except for the mixed-noise experiment, only one
noise family is applied at a time.

The mixed perturbation reflects the measurement characteristics of different
features. Continuous smartwatch measurements, including heart rate, oxygen
saturation, skin temperature, and blood pressure, are subject to additive noise
and drift. Cumulative measurements, such as step counts and active minutes, are
perturbed at the increment level and then reconstructed. For the
intelligent-insole data, pressure and gait-related continuous measurements are
perturbed multiplicatively, whereas ratio, score, and coordinate features are
perturbed additively with drift. Features with prescribed ranges are clipped
after perturbation, and discrete features in both modalities are kept
unchanged.
The complete elementwise definitions and implementation of the noise
generators are available in the accompanying repository
\url{https://github.com/DSVI2025/DSVI-HD}.

For the noise-robustness study, we fix a representative source user
$\bar u^S\in\mathcal U_S$ and evaluate the predictions on the ten test days
$d_j^S$, $j=1,\ldots,10$. For a noise family $\kappa$ and noise level $s$, let
$\widehat L_{\kappa,s}^{\bar u^S,j}(t_\nu)$ denote the predicted label on day
$d_j^S$. The corresponding accuracy is
\[
\operatorname{Acc}_{\kappa,s}(\bar u^S)
=
\frac{1}{10\cdot17280}
\sum_{j=1}^{10}
\sum_{\nu=1}^{17280}
\mathbf 1
\left\{
\widehat L_{\kappa,s}^{\bar u^S,j}(t_\nu)
=
L_S^{\bar u^S,j}(t_\nu)
\right\}
\times100\%.
\]

The unperturbed accuracy for this user is $97.338\%$.
Table~\ref{tab:application_noise} presents the accuracies under different noise
families and noise levels. The predictions remain stable under most individual
perturbations, whereas drift produces a larger loss of accuracy and mixed noise
causes the greatest degradation. Thus, in this case study, the
history-dependent model is relatively robust to isolated sensor perturbations
but is more sensitive to simultaneous perturbations of different types. Since
the experiment concerns one representative source user, the results are
intended to illustrate robustness rather than provide a population-level
assessment.

\begin{figure}[htbp]
\centering
\includegraphics[width=0.8\textwidth]{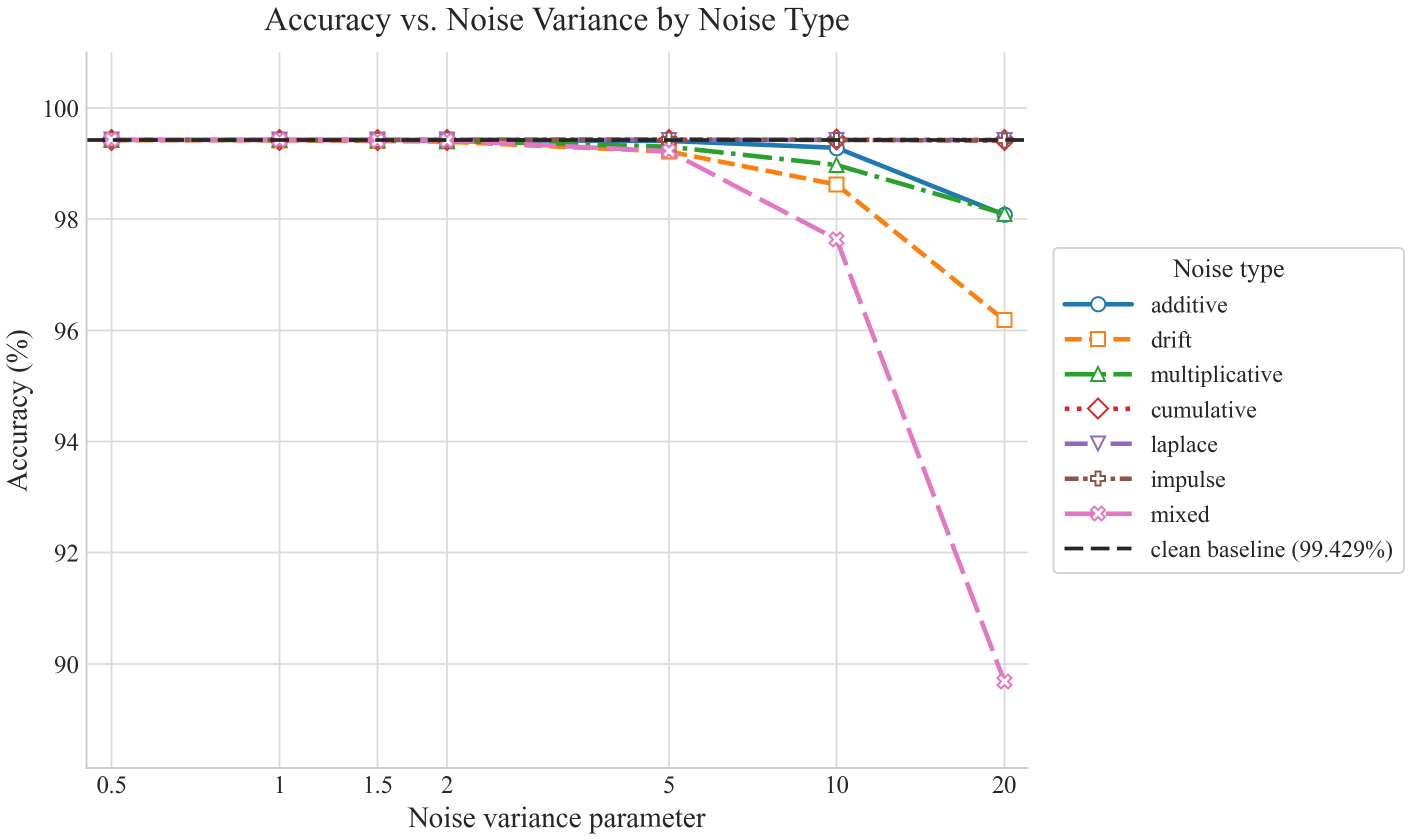}
\caption{Noise robustness for a representative held-out source-domain evaluation trajectory, with $\sigma_{\rm add}=0.025s$, $\sigma_{\rm drift}=0.015s$, $\sigma_{\rm mult}=0.030s$, $\sigma_{\rm cum}=0.040s$, and $s\in\{0.5,1,1.5,2,5,10,20\}$.}
\label{fig:application_noise_robustness}
\end{figure}

\begin{table}[htbp]
\caption{Prediction accuracy under different noise levels for one representative
held-out source-domain evaluation trajectory. The accuracy without noise is
\(97.338\%\).}
\label{tab:application_noise}
\centering
\small
\begin{tabular}{lccccccc}
\toprule
$s$ & Cumulative & Laplace & Impulse & Multiplicative & Additive & Drift & Mixed \\
\midrule
$5$  & 97.350 & 97.336 & 97.327 & 97.352 & 97.302 & 97.365 & 97.291 \\
$10$ & 97.357 & 97.336 & 97.299 & 97.174 & 97.191 & 97.161 & 96.082 \\
$20$ & 97.364 & 97.335 & 97.247 & 96.550 & 96.464 & 96.169 & 91.144 \\
\bottomrule
\end{tabular}
\end{table}

\subsection{Cross-domain transfer}
\label{subsec:application_transfer_learning}
We now investigate the cross-domain transfer of second-stage response trajectories across users. Recall that response computation and health-state label prediction proceed in two stages. First, second-stage response trajectories are derived from historical reference data. These trajectories are then combined with the current target state to generate health-state labels. In contrast, under the transfer-learning framework, source-domain response trajectories are precomputed offline and reused for target users, so online computation is limited to updating health-state label prediction.

The source and target cohorts comprise distinct individuals and contain no overlapping user trajectories. The experiment therefore assesses whether response trajectories learned from previously observed users can support online health-state prediction for new users, particularly when target-domain reference information is unavailable or costly to obtain. Our primary focus is how similarity-weighted transfer and  choice of refresh interval affect transfer performance and cost.

For each source user $u\in\mathcal U_S$, the
trajectories
$\mathbf y_{i,S}^{u,j}
(t_\nu,\bar{\xi}_{S,u}^{\iota,t_\nu})$
are computed from the source reference data and stored in advance. For a target
user $v\in\mathcal U_T$, the current target observation
$B_{i,T}^{v,j}(t_\nu)$ continues to enter the online state update, while the
precomputed source response is used in place of a response recomputed from the
target reference data. Thus, only the second-stage response trajectory is
transferred; the observations and health-state trajectory remain specific to
the target user.

\subsubsection{Similarity-weighted transfer learning}
\label{subsubsec:application_weighted_transfer}

We first compute a source--target distance at the user-pair level. For each pair
\((v,u)\!\in\!\mathcal U_T\times\mathcal U_S\), the distance is computed from the
first \(90\) reference days. Let \(\operatorname{Std}_{i}(\cdot)\) denote
columnwise standardization for modality \(i\), applied in the same way to the
corresponding source and target matrices. For \(i\!=\!1,2,3\), define

\vspace{-4pt}
\[
    d_i(v,u)
    :=
    \frac{1}{90\cdot 17280}
    \sum_{\nu=1}^{17280}
    \left\|
        \operatorname{Std}_{i}
        \left(
            \bigl(Z_{i,T}^{v}(t_\nu)\bigr)_{[1:90,:]}
        \right)
        -
        \operatorname{Std}_{i}
        \left(
            \bigl(Z_{i,S}^{u}(t_\nu)\bigr)_{[1:90,:]}
        \right)
    \right\|_F^2 .\vspace{-4pt}
\]
The overall source--target distance is
$
    d(v,u)
    :=2 d_1(v,u)+2 d_2(v,u)+ d_3(v,u).
$ The larger weights give more importance to the dynamic
smartwatch and insole modalities. A smaller \(d(v,u)\) indicates stronger
similarity between target user \(v\) and source user \(u\).
The distance is then converted into a transfer kernel by
$
    \kappa(v,u)
    :=
    \exp\{-d(v,u)/\tau(v)\}$ with $
    \tau(v)
    :=
    \operatorname{median}_{u'\in\mathcal U_S} d(v,u').
$

These weights are used to combine the precomputed source response trajectories.

With the  historical similarity score for each
target--source patient pair \((v,u)\in\mathcal U_T\times\mathcal U_S\), we can now combine several similar
source users and weight their response trajectories according to their
similarity to the target user
\cite{MansourMohriRostamizadeh2008MultipleSources,BlitzerEtAl2007DomainAdaptationBounds}.
For each target user $v\in\mathcal U_T$, let
$\operatorname{TopK}(v)\subseteq\mathcal U_S$ denote the set of $K$ source patients
with the smallest distances $d(v,u)$. For $u\in\operatorname{TopK}(v)$, define
$\omega^{(K)}(v,u):=
\kappa(v,u)/\sum_{u'\in\operatorname{TopK}(v)}\kappa(v,u')$.
During transfer, for target patient $v$, target evaluation day $d_j^T$, modality
$i$, and time $t_\nu$, we replace
$\mathbf y_{i,T}^{v,j}(t_\nu,\bar{\xi}_{T,v}^{\iota,t_\nu})$ by the weighted
source response
$$\widetilde{\mathbf y}_{i,T}^{v,j}(t_\nu):=
\sum_{u\in\operatorname{TopK}(v)}
\omega^{(K)}(v,u)\,
\mathbf y_{i,S}^{u,j}(t_\nu,\bar{\xi}_{S,u}^{\iota,t_\nu}).$$
The target observation $B_{i,T}^{v,j}(t_\nu)$ is still used in the online
state update.

\begin{figure}[htbp]
\centering
\includegraphics[width=0.7\textwidth]{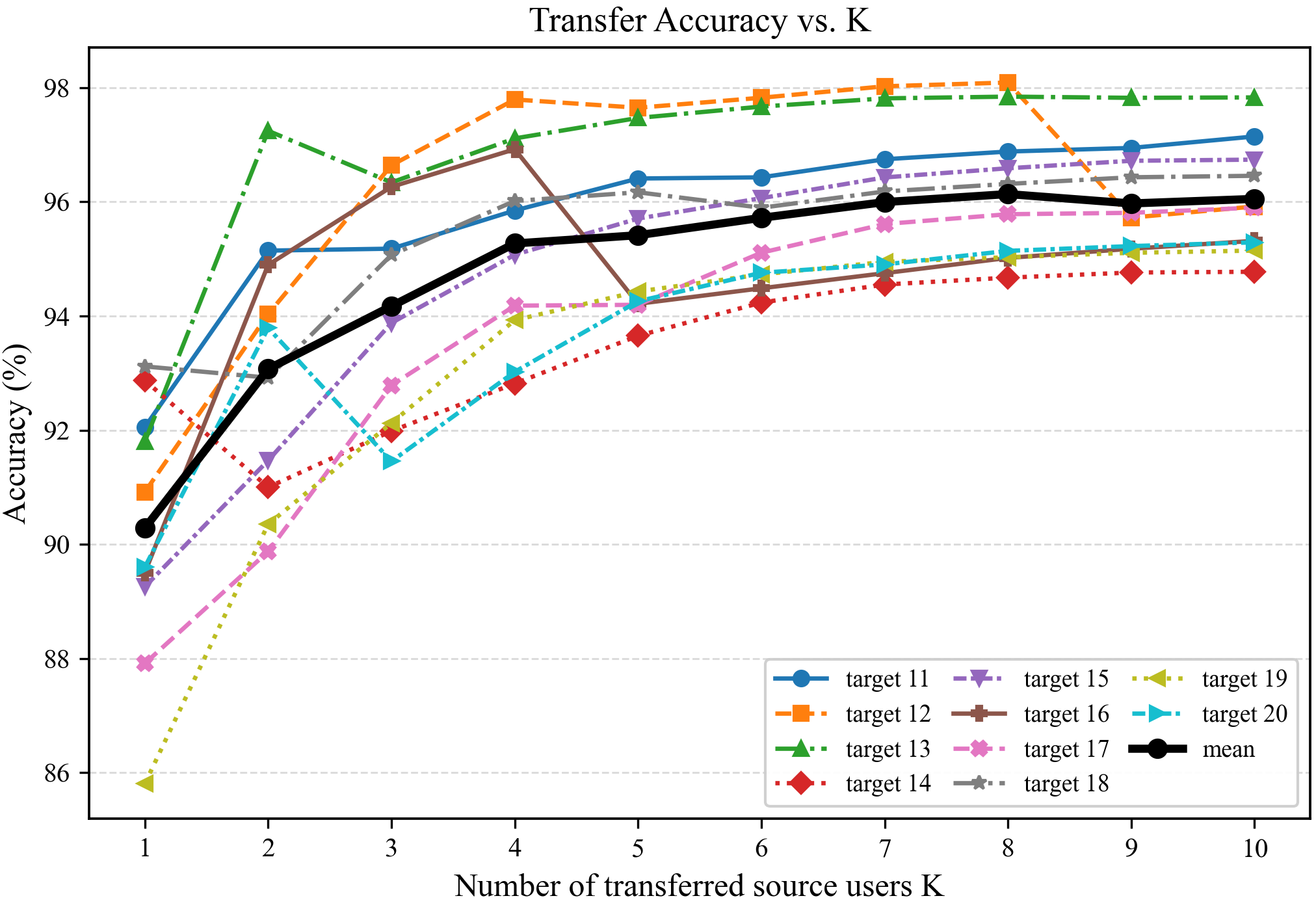}
\caption{Similarity-weighted transfer accuracy versus $K$. Colored curves denote target users; the black dashed curve denotes the mean.}
\label{fig:application_transfer_topk}
\end{figure}

\begin{figure}[htbp]
\centering
\includegraphics[width=0.7\textwidth]{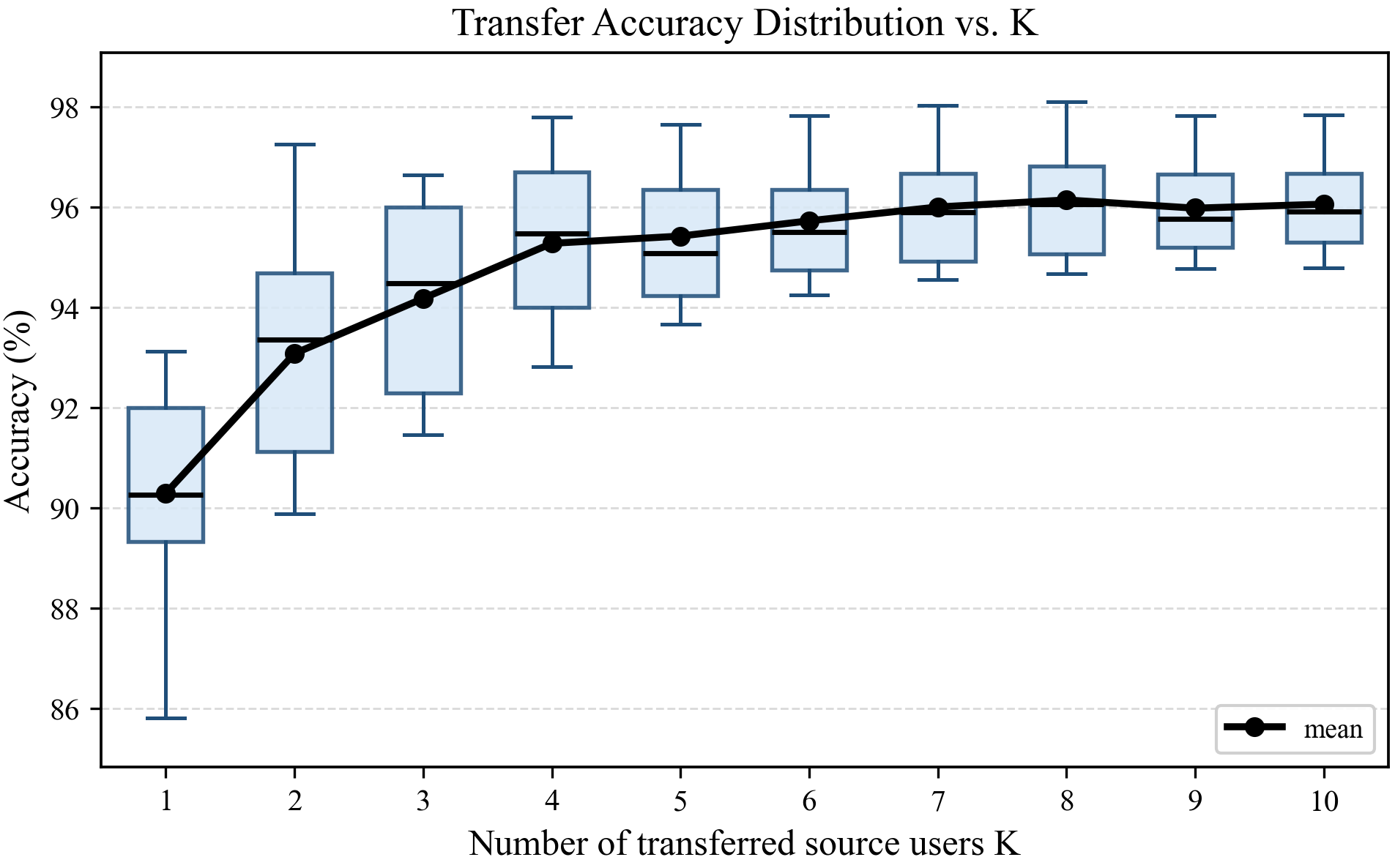}
\caption{Distribution of similarity-weighted transfer accuracy across target users for each $K$. Boxes show target-user accuracies, and the black curve shows the mean.}
\label{fig:application_transfer_topk_boxplot}
\end{figure}

Figure~\ref{fig:application_transfer_topk} and
Table~\ref{tab:application_topk} show that multi-source aggregation improves
transfer performance. The mean accuracy increases as \(K\) grows and reaches
its best value at \(K=8\), after which the improvement becomes marginal. This
result is close to the target-domain full-recomputation baseline, while avoiding
online response recomputation. The boxplot in
Figure~\ref{fig:application_transfer_topk_boxplot} further shows that
multi-source aggregation also improves the stability across target users.
Thus, the similarity score is more effective as a multi-source weighting signal
than as a single-source selection rule.

Table~\ref{tab:application_method_latency} further compares accuracy and online
latency. The weighted transfer method with \(K=8\) achieves an
accuracy of 96.614\% and a macro-F1 score of 0.967, which are close to those of
the full-recomputation baseline (97.928\% accuracy and 0.982 macro-F1), while
reducing the online batch runtime from 86.223 s to 0.620 s. This substantial latency reduction is
achieved by reusing precomputed source-domain response trajectories during the
online projected update, demonstrating the potential of transfer learning for
real-time health monitoring with limited computational resources.
The offline source-response
precomputation cost is not included in the reported online runtime.

\begin{table}[htbp]
\caption{Mean similarity-weighted transfer accuracy across ten target users.}
\label{tab:application_topk}
\centering
\setlength{\tabcolsep}{2.6pt}
\begin{tabular}{ccccccccccc}
\toprule
$K$ & 1 & 2 & 3 & 4 & 5 & 6 & 7 & 8 & 9 & 10\\
\midrule
Accuracy (\%) &
91.008 & 94.178 & 95.292 & 96.012 & 96.224 &
96.288 & 96.394 & 96.614 & 96.327 & 96.444\\
\bottomrule
\end{tabular}
\end{table}

\begin{table}[htbp]
\caption{Method performance and online latency. Per-point latency is measured at the five-second update scale with precomputed source-domain responses.}
\label{tab:application_method_latency}
\centering
\setlength{\tabcolsep}{2.2pt}
\begin{tabular}{lccccc}
\toprule
Method & Accuracy & Macro-F1 & Point ms & Batch s & Speedup\\
\midrule
Target-domain full-recomputation baseline & 97.928 & 0.982 & 0.499 & 86.223 & 1.0\\
Random-source transfer   & 91.139 & 0.904 & 0.00360 & 0.622 & 138.7\\
Nearest-source transfer  & 91.008 & 0.904 & 0.00361 & 0.623 & 138.3\\
Weighted transfer, $K=8$ & 96.614 & 0.967 & 0.00359 & 0.620 & 139.0\\
\bottomrule
\end{tabular}
\end{table}

\subsubsection{Delayed response reuse}
\label{subsec:application_temporal_transfer}

We next examine delayed reuse of the transferred second-stage response. The
target observations are still evaluated on the original sampling grid, but the
transferred source responses are refreshed only at selected time indices and
then reused between two refreshes. This gives a coarser version of the original
\(K=8\) similarity-weighted transfer rule and reduces the online refresh cost.
We use the same source-to-target setting and the same \(K=8\) transfer weights
as in Subsection~\ref{subsec:application_transfer_learning}. Let \(a\) denote
the delay parameter, measured in units of the sampling interval. For \(a=0\),
the response is refreshed at every time point. For \(a\ge1\), define the most
recent refresh index before or at \(t_\nu\) by
$
    r_a(\nu)
    =
    1+a\left\lfloor\frac{\nu-1}{a}\right\rfloor .
$
For \(a=0\), we set \(r_0(\nu)=\nu\).
Then the delayed \(K=8\) transferred response for target user
\(v\in\mathcal U_T\), target evaluation day \(d_j^T\), modality \(i\), and time
\(t_\nu\) is defined by
$
\widetilde{\mathbf y}_{i,T,a}^{v,j}(t_\nu)
=
\sum_{u\in\operatorname{TopK}(v)}
\omega^{(8)}(v,u)\,
\mathbf y_{i,S}^{u,j}
(
t_{r_a(\nu)},
\bar{\xi}_{S,u}^{\iota,t_{r_a(\nu)}}
).
$
Thus, increasing \(a\) reuses older source responses over longer time intervals.
The relative refresh cost is defined by
\[
\operatorname{RecompCost}(a)
=
\begin{cases}
1, & a=0,\\
1/a, & a\ge1,
\end{cases}
\quad
\operatorname{CostReduction}(a)=1-\operatorname{RecompCost}(a).
\]
{The delayed-update strategy further reduces the frequency of updating the
transferred responses while leaving the online label-prediction procedure unchanged.}

Figure~\ref{fig:application_temporal_lag} 
shows the trade-off between the response
refresh interval and target prediction accuracy. When \(a=0\), the response is
refreshed at every sampling point, recovering  \(K=8\)
weighted-transfer result in
Subsection~\ref{subsec:application_transfer_learning}. 
Table~\ref{tab:application_lag} summarizes the corresponding relative costs and
mean accuracies:  short refresh intervals can substantially
reduce the cost of recomputing transferred responses with little loss of
accuracy, whereas large refresh intervals eventually degrade target prediction.
In particular, the weighted transfer rule preserves high prediction accuracy
under moderate refresh delays, while the gradual accuracy degradation with
longer intervals indicates the importance of timely response updates.

\begin{figure}[htbp]
\centering
\includegraphics[width=0.7\textwidth]{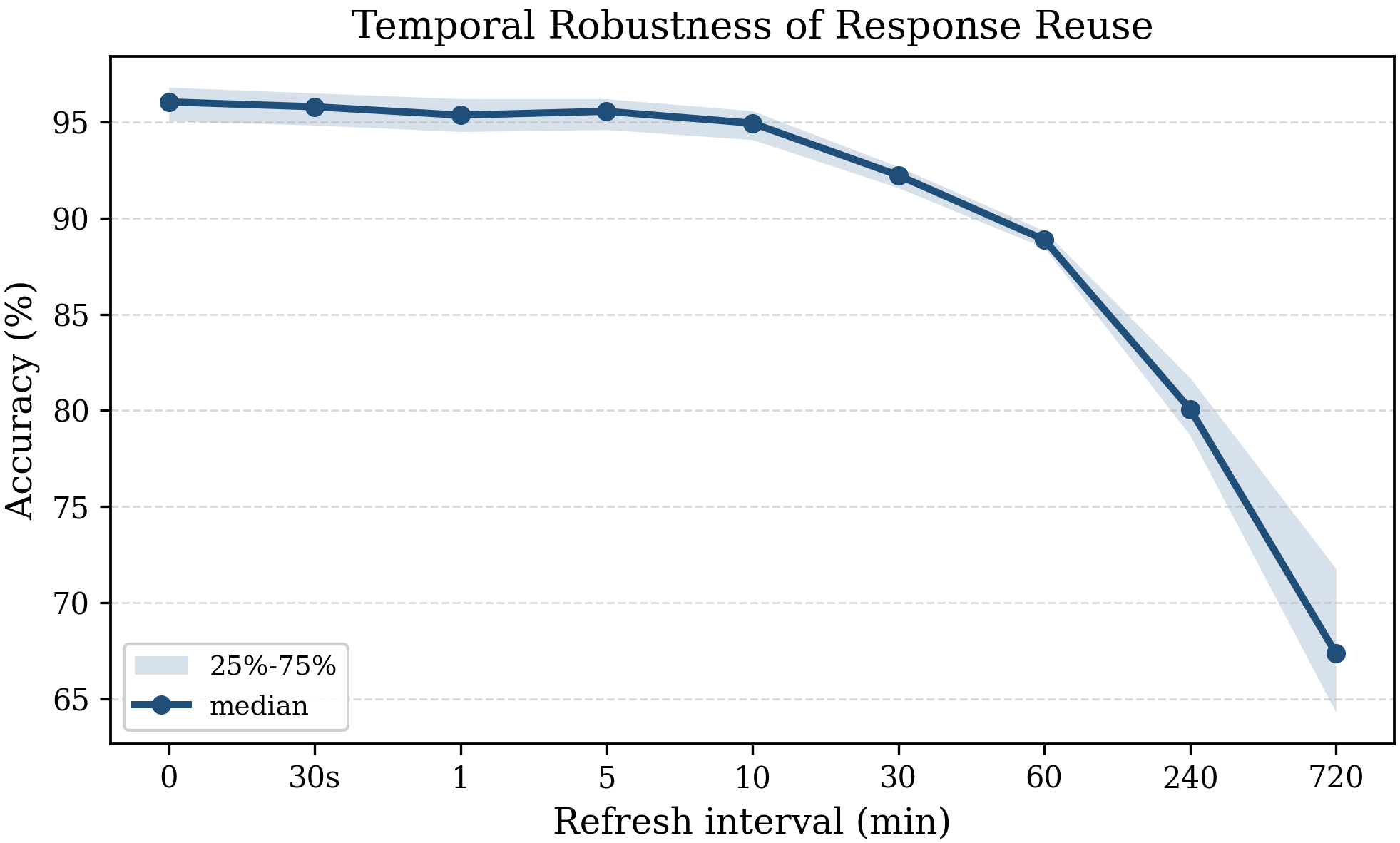}
\caption{Temporal robustness of source-to-target \(K=8\) response reuse. The
line shows median accuracy, and the shaded band shows the 25th--75th percentile
range.}
\label{fig:application_temporal_lag}
\end{figure}

\begin{table}[htbp]
\centering
\caption{Prediction accuracy and recomputation cost under delayed reuse of
transferred responses for the source-to-target \(K=8\) weighted transfer rule.}
\label{tab:application_lag}
\resizebox{\textwidth}{!}{%
\begin{tabular}{lccccccccc}
\toprule
Refresh interval (min)
& 0 & 0.5 & 1 & 5 & 10 & 30 & 60 & 240 & 720 \\
\midrule
Relative refresh cost (\%)
& 100.000 & 16.667 & 8.333 & 1.667 & 0.833 & 0.278 & 0.139 & 0.035 & 0.012 \\
Mean accuracy (\%)
& 96.614 & 96.503 & 96.274 & 95.680 & 95.054 & 93.855 & 92.499 & 89.582 & 81.990 \\
\bottomrule
\end{tabular}%
}
\end{table}

In summary, the numerical findings support both the predictive effectiveness and the computational value of the proposed DSVI framework; its history-dependent response trajectories can be robustly transferred and reused for high accuracy and low latency, especially with  limited target-domain data or online computational resources. 

\begin{itemize}
\item On held-out source-domain data, the DSVI model achieves  high F1-scores (0.9848, 0.9744, and 0.9974 for the healthy, weak, and ill classes, respectively), demonstrating consistently strong discrimination across health states. The perturbation experiments further indicate that predictions are stable under most isolated sensor-noise types. 
\item Regarding the feasibility and effectiveness of transfer learning between DSVI systems,  similarity-weighted transfer using the responses of the most similar source users attains 0.97 
accuracy and a macro-F1 score of 0.967, close to the full target-domain recomputation benchmark, while reducing the online batch runtime from 86.223 seconds to 0.620 seconds, an approximately 139-fold speedup. 
Meanwhile,  an appropriate choice of the refresh interval for the transferred source-domain responses preserves a mean accuracy of 0.963, with only about 8 percent of the full refresh cost.

\end{itemize}

\section{Conclusion}
\label{sec:conclusion}

This paper develops a history-dependent DSVI framework for dynamic systems driven by exogenous random
processes. 
The regularity and measurability of the second-stage response
and the well-posedness of the first-stage state trajectory are established, along with a
sample average approximation scheme and its uniform convergence.  A local \(1/2\)-Hölder response estimate for parametric VIs with
moving feasible sets is also derived, yielding a stability bound for the
first-stage trajectory. The stability analysis of this history-dependent DSVI supports transfer learning and the efficient reuse of response trajectories between two similar DSVI systems. 
Numerical experiments 
on the elderly-health monitoring system illustrate 
the feasibility and stability of transfer learning,
which can achieve an accuracy close to the target-domain
full-recomputation baseline with a significant reduction in computational cost, even with limited data. Future work may consider more general models with weaker assumptions  and extensions with multimodal datasets.

\bibliographystyle{siamplain}
\bibliography{Transfer_reference}

\end{document}